\documentclass[a4paper,11pt]{amsart}
\usepackage[utf8x]{inputenc}

\usepackage{amssymb,amsfonts}
\usepackage{amsmath,amsthm}
\usepackage[all,arc, cmtip]{xy}
\usepackage{enumerate}
\usepackage{paralist}
\usepackage{booktabs} 
\usepackage{array} 
\usepackage{mathrsfs}
\usepackage{url}
\usepackage{color}
\usepackage{cancel}
\usepackage{todonotes} 
\usepackage{hyperref} 
\usepackage{mathtools}
\usepackage[mathscr]{euscript}
\usepackage[nameinlink,capitalise,noabbrev]{cleveref}
\usepackage[shortcuts]{extdash}

\usepackage{color}

\usepackage{tikz-cd}

\usepackage{hyphenat}

\newcommand{\binr}[2]{\ar@<.5ex>[r]^{#1} \ar@<-.5ex>[r]_{#2}}
\newcommand{\bind}[2]{\ar@<.5ex>[d]^{#1} \ar@<-.5ex>[d]_{#2}}
\newdir{ >}{{}*!/-5pt/\dir{>}}
\newdir^{ (}{{}*!/-5pt/@^{(}}

\def\ra{\rightarrow}

\def\id{{\rm id}}

\def\ZZ{{\mathbb{Z}}}
\def\NN{{\mathbb{N}}}

\def\VV{{\mathbb{V}}}

\def\cA{{\mathcal A}}

\def\cC{{\mathcal C}}
\def\cE{{\mathcal E}}
\def\cM{{\mathcal M}}

\def\cO{{\mathcal O}}
\def\cP{{\mathcal P}}

\newcommand{\Sub}{\mathrm{Sub}}
\newcommand{\Homm}{\mathrm{Hom}}

\newcommand{\sch}{\textnormal{Sch}}
\newcommand{\schs}{\sch_S}
\newcommand{\sms}{\mathrm{Sm}_S}
\newcommand{\Ho}{\textnormal{Ho}}
\newcommand{\pre}{\textnormal{\textbf{Pre}}}
\newcommand{\spre}{\textnormal{s\textbf{Pre}}}

\newtheorem{thm}{Theorem}[section]
\newtheorem*{thm*}{Theorem}

\newtheorem{cor}[thm]{Corollary}
\newtheorem{prop}[thm]{Proposition}
\newtheorem{lem}[thm]{Lemma}

\theoremstyle{definition}

\newtheorem{defn}[thm]{Definition}
\newtheorem*{defn*}{Definition}

\newtheorem{rem}[thm]{Remark}

\theoremstyle{remark}

\makeatletter
\let\c@equation\c@thm
\makeatother
\numberwithin{equation}{section}
\title[]
{}

\title[Comparison of Exterior Power Operations]
{Comparison of Exterior Power Operations on Higher \texorpdfstring{$K$}{K}-Theory of Schemes}

\author{Bernhard K\"{o}ck}
\address{School of Mathematical Sciences\\ University of Southampton\\  Southampton SO17 1BJ\\ United Kingdom}
\email{B.Koeck@soton.ac.uk}

\author{Ferdinando Zanchetta}
\address{FABIT\\ University of Bologna\\ Via San Donato 15\\ Bologna 40127\\ Italy}
\email{ferdinando.zanchett2@unibo.it}

\date{\today}

\begin{document}

\begin{abstract} Exterior power operations provide an additional structure on $K$-groups of schemes which lies at the heart of Grothendieck's Rie\-mann-Roch theory. Over the past decades, various authors have constructed such operations on higher $K$-theory. In this paper, we prove that these constructions actually yield the same operations, ultimately matching up the explicit combinatorial  description by Harris, the first author and Taelman on the one hand and the recent, conceptually clear-cut construction by Barwick, Glasman, Mathew and Nikolaus on the other hand. This also leads to the proof of a conjecture by the first author about composition of these operations in the equivariant context, completing the proof that higher equivariant $K$-groups satisfy all axioms of a $\lambda$-ring. 
\end{abstract}

\maketitle

\section*{Introduction}

\subsection*{Background} Tensor products and exterior powers furnish the exact category of finitely generated projective modules over a commutative ring or, more generally, the category of locally free modules of finite rank on a scheme with additional structure which has been studied in great depth in countless articles and books. For example, the main purpose of the book \cite{FL} by Fulton and Lang  is to explain Grothendieck's tremendous  and rather formal ``machine'' which turns this structure into Riemann-Roch type theorems and which has seen the light in the seminal \cite{SGA6}. In particular, exterior power operations on $K$-groups play a fundamental role in Algebraic Geometry and beyond.

While the construction of exterior power operations on Grothendieck groups is fairly elementary, the construction of these operations on higher $K$-groups is more delicate. Over the past decades and especially more recently, many authors have provided such constructions, of varying flavour and generality.

\subsection*{Main results} The core object of this paper is to prove for most of these constructions that they actually yield the same operations. This will culminate in the proof of a conjecture \cite[Conjecture~(2.7)]{GRR} by the first author about composition of exterior power operations in the equivariant context;  this conjecture was the last missing ingredient needed to prove that Grayson’s operations define a $\lambda$-ring structure on the higher equivariant $K$-groups (see Section 5).

The construction given by Harris, the first author and Taelman in \cite{HKT17} is different from all other constructions as it defines exterior power operations on higher $K$-groups purely combinatorially and does not resort to any homotopy theoretical methods. This construction has been made possible by Grayson's groundbreaking description of higher $K$-groups in terms of generators and relations, see \cite[Corollary~7.4]{Grayson2012}. The bulk of this paper deals with showing that these operations agree with the operations defined by Grayson in \cite{GraysonExterior}: this is accomplished in \cref{Thm: main thm}. In \cref{thm: comparison with Zanchetta and Riou}, we will show furthermore that Grayson's operations agree with the operations recently defined by the second author in \cite{Zan21}, which in turn generalise the operations defined by Riou in \cite{RRR} from non-singular schemes to fairly arbitrary schemes. This will finally allow us to show in \cref{cor: comparing with Gillet-Soule} that Soul\'e's construction in \cite{Sou85} and the more general construction by Gillet and Soul\'e in \cite{GS} yield the same operations. In the same corollary we will furthermore show that the very recent construction by Barwick, Glasman, Mathew and Nikolaus in \cite{Barwick} for polynomial functors on stable $\infty$-categories also yields the same operations when applied to exterior power functors on the category of perfect complexes.

\subsection*{Related results} We now mention some further papers about operations on higher $K$-theory. First, in \cite{JP}, Joshua and Pelaez use techniques from various aforementioned papers to construct exterior operations on the $K$-theory of (smooth) algebraic stacks. In \cite{HK}, Harris and the first author have provided some early evidence that the operations defined in \cite{HKT17} agree with those defined by Hiller in \cite{Hiller} for affine schemes. Grayson has already shown in \cite{GraysonExterior} that his operations agree with Hiller's and those defined by Kratzer in \cite{Kratzer}. The aforementioned construction by Soul\'e and Levine's construction in \cite{Levine} generalise Hiller's and Kratzer's. Nenashev has shown in \cite{NenComparison} that the operations he defined in \cite{NenLambda} agree with Grayson's operations as well.

\subsection*{Overview} In order to prove \cref{Thm: main thm}, we first extend the notion of exterior powers of modules to exterior powers of various types of complexes that appear in the proof of Grayson's explicit description of higher $K$-groups. We do this in \cref{Sec: Operations} in the axiomatic and more general context of abstract power operations introduced in \cite{GraysonExterior}. To define the $k^\mathrm{th}$ exterior power of a complex we use the Dold-Puppe construction from \cite{DP}, i.e., we apply the $k^\mathrm{th}$ exterior power functor to the simplicial object corresponding to the complex via the Dold-Kan correspondence. A crucial detail here is that the $k^\mathrm{th}$ exterior power of a quasi-isomorphism between complexes of locally free modules is again a quasi-isomorphism, see \cref{prop: exterior powers of quasi-isom}.

In \cref{sec: operations on spaces}, we then turn the operations defined in \cref{Sec: Operations} into operations on the corresponding $K$-theory spaces and $K$-groups. To this end, we generalise Grayson's construction in \cite{GraysonExterior} using tools for categories with weak equivalences provided by Waldhausen in \cite{Wald}, by Gunnarson, Schw\"anzl, Vogt and Waldhausen in \cite{GSVW} and by Gunnarson and Schw\"anzl in \cite{GunnSchw}.

\cref{sec: comparing lambda operations} deals with proving our main theorem, \cref{Thm: main thm}. Recall that Grayson's explicit description of higher $K$-groups in \cite[Corollary~7.4]{Grayson2012} is
shown via a composition of a number of insightful homotopy equivalences concerning the $K$-theory spaces
of various exact categories of different types of complexes. The main idea then
is just to prove that each of these homotopy equivalences is compatible with exterior power operations. However, Grayson's homotopy equivalences involve spectra and homotopy cofibres on which exterior powers seem not to exist. The more sophisticated part in this approach is then to first show that natural exterior power operations on the various $K$-theory spaces (not spectra) appearing in many of these homotopy equivalences do exist; this has been the purpose of Sections~1 and~2. Finally, we replace those arguments that involve homotopy cofibres with arguments in the more elementary context of homotopy groups (of homotopy fibres of maps between spaces).

Thus, our proof of \cref{Thm: main thm} shows that the language of spectra is actually not needed for the proof of \cite[Corollary~7.4]{Grayson2012}. In particular, as a by-product, we render Grayson's proof somewhat more elementary, making it accessible to readers not familiar with the language of spectra. However, we do not simplify the underlying reasoning, and obtaining the relevant statements in \cite{Grayson2012} (such as the epitomising Corollary~6.5 there) not just for $K$-groups but for the underlying $K$-theory spectra is of independent interest.

In \cref{sec: comparison with other operations}, we match up Grayson's exterior power operations on schemes with further operations defined previously. The crucial new tool here is the main theorem of \cite{Zan21} which in a sophisticated sense says that we have a bijection between the set of sufficiently functorial operations on higher $K$-theory of fairly general schemes and the set of operations on the Grothendieck group functor~$K_0$. This generalisation of a theorem of Riou \cite{RRR} yields another construction of exterior power operations on higher $K$-theory. At the same time it allows us to verify that these operations agree on the one hand with the operations defined by Grayson, see \cref{thm: comparison with Zanchetta and Riou}, and on the other hand with the operations defined by Soul\'e in \cite{Sou85}, with the more general operations defined by Gillet and Soul\'e in \cite{GS} and with the operations defined by Barwick, Glasman, Mathew and Nikolaus in \cite{Barwick}, see \cref{cor: comparing with Gillet-Soule}, respectively.

In \cref{sec: Equivariant K-theory}, we turn to higher \emph{equivariant} $K$-theory, more precisely, to the $K$-theory of locally free $G$-modules on a scheme equipped with the action of a flat group scheme $G$. In \cite{GRR}, the first author used exterior powers of $G$-modules and ``Grothendieck's machine'' to formulate and prove equivariant Riemann-Roch type theorems variants of which have later been used for stacks (e.g., see \cite{KS20} and the literature cited there). In particular, the statements in \emph{loc.\ cit.}\ depend on equivariant $K$-theory satisfying some or all axioms of a $\lambda$-ring. These axioms have been proved in \emph{loc.\ cit.}\ for $K_0$ and some of them also for higher $K$-groups, but the axiom for composition of exterior power operations is a conjecture there. \cref{thm: conjecture} now proves this conjecture. More precisely, our main theorem (\cref{Thm: main thm}) reduces it to proving it for the combinatorially defined operations in \cite{HKT17}, when it becomes verifiable by the methods developed in \cite[Section~8]{HKT17}.

The paper finishes with two little appendices which prove some technical facts that are used in the proof of \cref{Thm: main thm}. \cref{Sec: non-negative complexes} verifies that using non-negatively supported complexes rather than arbitrary bounded complexes as in \cite{GraysonExterior} doesn't result in any difference for $K$-theory. \cref{Sec: Homotopy} gives an elementary criterion for a map induced on homotopy fibres to agree with an intrinsically defined map.

\medskip

\subsection*{Acknowledgements} The authors wish to thank Clark Barwick for early and insightful discussions related to the topic of this paper. Furthermore, they thank Akhil Mathew for explaining various aspects of \cite{Barwick} and they thank Tom Harris, Jens Hornbostel and Marco Schlichting for their continued and stimulating interest. Finally, they thank the referees for valuable suggestions improving the readability of the paper.

\section{Operations on Categories of Complexes}\label{Sec: Operations}

The object of this section is to introduce (exterior) power operations on various categories of complexes and binary complexes that appear in the proof of Grayson's description of higher $K$-groups, see \cite{Grayson2012}. We'll do this in the context of the axiomatic setup of (exterior) power operations introduced in  \cite[Section~7]{GraysonExterior}. The following definition reproduces this axiomatic setup.

\begin{defn}\label{def: power operations}
An {\em assembly of power operations} consists of the following data subject to the conditions detailed below. \\
Let $\cM_n$, $n \ge 0$, be a sequence of exact categories. We assume we are given bi-exact functors
\[\otimes \colon \cM_n \times \cM_p \rightarrow \cM_{n+p}, \quad \textrm{ for } n, p \ge 0,\]
called {\em tensor products}, which are associative in the obvious sense. For $k \ge 1$ and $n \ge 0$, let $F_k(\cM_n)$ denote the category of sequences $V_1 \rightarrowtail \ldots \rightarrowtail V_k$ of admissible monomorphisms in $\cM_n$. We assume we are given functors
\[ F_k(\cM_n) \rightarrow \cM_{nk}, \quad (V_1 \rightarrowtail \ldots \rightarrowtail V_k) \mapsto V_1 \wedge \ldots \wedge V_k, \quad \textrm{ for } k \ge 1, n\ge 0,\]
called {\em power operations}. (In particular, these functors are assumed to be the identity if $k=1$.) We finally assume: \\
(E1, E2) Given $V \rightarrowtail \ldots \rightarrowtail W \rightarrowtail X \rightarrowtail \ldots \rightarrowtail Y$, there are natural maps
\[V \wedge \ldots \wedge W \otimes X \wedge \ldots \wedge Y \rightarrow V \wedge \ldots \wedge W \wedge X \wedge \ldots \wedge Y\]
and

\[V \wedge \ldots \wedge W \wedge X \wedge \ldots \wedge Y \rightarrow V \wedge \ldots \wedge W \otimes \frac{X}{W} \wedge \ldots \wedge \frac{Y}{W}\]
(for any choice of quotient objects $\frac{X}{W}, \ldots, \frac{Y}{W}$). These maps are assumed to be associative in the obvious sense.\\
(E3, E4) Given $U \rightarrowtail \ldots \rightarrowtail V \rightarrowtail W \rightarrowtail \ldots \rightarrowtail X \rightarrowtail Y \rightarrowtail \ldots \rightarrowtail Z$, the following two diagrams commute:
\[\resizebox{\hsize}{!}{ $ \xymatrix{
U \wedge \ldots \wedge V \wedge W \wedge \ldots \wedge X \otimes Y \wedge \ldots \wedge Z \ar[r] \ar[d] &
U \wedge \ldots \wedge V \wedge W \wedge \ldots \wedge X \wedge Y \wedge \ldots \wedge Z \ar[d]\\
U \wedge \ldots \wedge V \otimes \frac{W}{V} \wedge \ldots \wedge \frac{X}{V} \otimes \frac{Y}{V} \wedge \ldots \wedge \frac{Z}{V} \ar[r] &
U \wedge \ldots \wedge V \otimes \frac{W}{V} \wedge \ldots \wedge \frac{X}{V} \wedge \frac{Y}{V} \wedge \ldots \wedge \frac{Z}{V}
}$}\]
and
\[ \resizebox{\hsize}{!}{$\xymatrix{
U \wedge \ldots \wedge V \otimes W \wedge \ldots \wedge X \wedge Y \wedge \ldots \wedge Z \ar[r] \ar[d] &
U \wedge \ldots \wedge V \wedge W \wedge \ldots \wedge X \wedge Y \wedge \ldots \wedge Z \ar[d]\\
U \wedge \ldots \wedge V \otimes {W} \wedge \ldots \wedge {X} \otimes \frac{Y}{X} \wedge \ldots \wedge \frac{Z}{X} \ar[r] &
U \wedge \ldots \wedge V \wedge {W} \wedge \ldots \wedge {X} \otimes \frac{Y}{X} \wedge \ldots \wedge \frac{Z}{X}
}$}\]
(E5) Given $U \rightarrowtail \ldots \rightarrowtail V \rightarrowtail W' \rightarrowtail W \rightarrowtail X \rightarrowtail \ldots \rightarrowtail Y$, the following sequence is exact:
{\footnotesize\[ \begin{aligned}
0 \rightarrow U \wedge \ldots \wedge V \wedge {W'} \wedge {X} \wedge  \ldots \wedge Y &
\rightarrow U \wedge \ldots \wedge V \wedge {W} \wedge {X} \wedge  \ldots \wedge Y\\
&\rightarrow  U \wedge \ldots \wedge V \otimes \frac{W}{W'} \wedge \frac{X}{W'} \wedge  \ldots \wedge \frac{Y}{W'} \rightarrow 0
\end{aligned}\]}
\end{defn}

We recall, for any assembly of power operations as above, elementary standard constructions yield products and power operations on Grothendieck groups,
\[K_0(\cM_n) \times K_0(\cM_p) \ra K_0(\cM_{n+p}) \quad \textrm{and} \quad \lambda^k \colon K_0(\cM_n) \ra K_0(\cM_{nk})\]
(for $n,p \ge 0$ and $k \ge 1$), and the object of \cite{GraysonExterior} is to extend these operations to higher $K$-groups.

Our main example of an assembly of power operations is of course the case when every $\cM_n$ is equal to the exact category $\cP(X)$ of locally free $\cO_X$-modules of finite rank on a quasi-compact scheme $X$, the functor $\otimes$ is the usual tensor product and $V_1 \wedge \ldots \wedge V_k$ is defined to be the image of the tensor product $V_1 \otimes \ldots \otimes V_k$ in the $k^\mathrm{th}$ exterior power $\bigwedge^k V_k$ (for any $k \ge 1$ and any sequence $V_1 \rightarrowtail \ldots \rightarrowtail V_k$ in $\cP(X)$). We will call this the \emph{standard assembly of power operations on $\cP(X)$}. But we could for example also replace the $k^\mathrm{th}$ exterior power with the $k^\mathrm{th}$ symmetric power. Further important examples are contained in \cref{sec: Equivariant K-theory}, in \cite[Section 10]{GraysonExterior} and in \cite[Section 3]{GunnSchw}, including some where the categories $\cM_n$ are not equal to each other.

\medskip

Let $\cM$ be an exact category. Slightly deviating from standard notations, let $C^\infty \cM$ denote the category of chain complexes in $\cM$ supported in non-negative degrees and let $C\cM$ denote the sub-category consisting of complexes which are in addition bounded. If $\cM$ is idempotent complete, the Dold-Kan correspondence (see \cite[Theorem 1.2.3.7]{HigherAlg} for a version which is general enough for our purposes) gives an equivalence between the category $C^\infty \cM$ and the category of simplicial objects in $\cM$. Let $N$ and $\Gamma$ denote the corresponding mutually inverse functors where $N$ is the normalized chain complex functor.

\medskip

For the rest of this section we assume we are given an assembly of power operations as in \cref{def: power operations}. We furthermore assume that every category~$\cM_n$ is idempotent complete.

\medskip

The following definitions are variants of the constructions in \cite[Sections 3 and 5]{HKT17}.

For $n, p \ge 0$, let
\[ \otimes_\Delta \colon C^\infty \cM_n \times C^\infty \cM_p \rightarrow C^\infty \cM_{n+p}\]
denote the simplicial tensor product as defined in \cite[Section~5]{HKT17}. We recall it maps a pair of complexes $V., W.$ to $N(\Gamma V. \otimes \Gamma W.)$ where, by abuse of notation, we write $\Gamma V. \otimes \Gamma W.$ also for the diagonal of the a priori bi-simplicial object $\Gamma V. \otimes \Gamma W.$ in $\cM_{n+p}$.\\
Now, let $V_{1}. \rightarrowtail \ldots \rightarrowtail V_k.$ be an object in $F_k(C^\infty \cM_n)$. We consider the resulting sequence $\Gamma V_1. \rightarrowtail \ldots \rightarrowtail \Gamma V_k.$ of simplicial objects in $\cM_n$ as a simplicial object in $F_k(\cM_n)$, then compose with the given power operation functor $F_k(\cM_n) \rightarrow \cM_{nk}$ to obtain the simplicial object $\Gamma V_1. \wedge \ldots \wedge \Gamma V_k.$ in $\cM_{nk}$, and finally define
\[V_1. \wedge \ldots \wedge V_k. := N (\Gamma V_1. \wedge \ldots \wedge \Gamma V_k.) \quad \textrm{ in } \quad  C^\infty \cM_{nk}.\]

\begin{prop}\label{prop: DoldPuppe Operations}
The sequence of exact categories $C^\infty \cM_n$, $n \ge 0$, together with the functors defined above is again an assembly of power operations. By restricting to the subcategories $C\cM_n$, $n \ge 0$, we obtain another assembly of power operations.
\end{prop}

\begin{proof}
The second statement follows from the first statement by \cite[Lemma~5.6]{HKT17} and \cite[Corollary~4.6]{SatkKoe}. To prove the first statement we observe that there is an assembly of power operations on the category of simplicial objects in the categories $\cM_n$ and that we may transport this to $C^\infty\cM_n$ using the Dold-Kan correspondence.
To explain this argument in detail, we include, as an example, the proof of axioms (E2) and~(E5).\\
For (E2), let $V. \rightarrowtail \dots \rightarrowtail W. \rightarrowtail X. \rightarrowtail \ldots \rightarrowtail Y.$ be a sequence of (levelwise) admissible monomorphisms between complexes. Applying the functor $\Gamma$ we obtain the sequence
\[\Gamma V. \rightarrowtail \ldots \rightarrowtail \Gamma W. \rightarrowtail \Gamma X. \rightarrowtail \ldots \rightarrowtail \Gamma Y\]
of levelwise admissible mnonorphisms between simplicial objects.
We now apply (E2) for $\cM_n$, $n \ge 0$, levelwise. The maps obtained assemble to a simplical map
\[\Gamma V. \wedge \ldots \wedge \Gamma W. \wedge \Gamma X. \wedge \ldots \wedge \Gamma Y. \rightarrow \Gamma V. \wedge \ldots \wedge \Gamma W. \otimes \Gamma \frac{X.}{W.} \wedge \ldots \wedge \Gamma \frac{Y.}{W.}.\]
Finally, we apply the functor $N$ and obtain the required map
\[V. \wedge \ldots \wedge W. \wedge  X. \wedge \ldots \wedge  Y. \rightarrow V. \wedge \ldots \wedge  W. \otimes_\Delta  \frac{X.}{W.} \wedge \ldots \wedge  \frac{Y.}{W.}.\]
These maps are obviously natural and associative again. \\
For (E5), let $U. \rightarrowtail \ldots \rightarrowtail V. \rightarrowtail W'. \rightarrowtail W. \rightarrowtail X. \rightarrowtail \ldots \rightarrowtail Y.$ be a sequence of complexes. From the constructions above, we obtain the sequence of simplicial objects:
\[\resizebox{0.95\hsize}{!}{$\begin{aligned}
0 \rightarrow \Gamma U. \wedge \ldots \wedge \Gamma V. \wedge \Gamma W'. \wedge  \Gamma X. \wedge & \ldots \wedge \Gamma Y.
\rightarrow \Gamma U. \wedge \ldots \wedge \Gamma V. \wedge \Gamma W. \wedge \Gamma X. \wedge  \ldots \wedge \Gamma Y.\\
&\rightarrow  \Gamma U. \wedge \ldots \wedge \Gamma V. \otimes \frac{\Gamma W.}{\Gamma W'.} \wedge \frac{\Gamma X.}{\Gamma W'.} \wedge  \ldots \wedge \frac{\Gamma Y}{\Gamma W'.} \rightarrow 0
\end{aligned}$}\]
This sequence is exact levelwise (by (E5) for $\cM_n$, $n \ge 0$). As the normalized chain complex functor $N$ is a direct summand of the (ordinary) associated chain complex functor (see paragraph before \cite[Proposition~1.2.3.17]{HigherAlg}), the resulting sequence
{\footnotesize\[ \begin{aligned}
0 \rightarrow U. \wedge \ldots \wedge V. \wedge {W'.} \wedge {X.} \wedge  \ldots \wedge Y. &
\rightarrow U. \wedge \ldots \wedge V. \wedge {W.} \wedge {X.} \wedge  \ldots \wedge Y.\\
&\rightarrow  U. \wedge \ldots \wedge V. \otimes_\Delta \frac{W.}{W'.} \wedge \frac{X.}{W'.} \wedge  \ldots \wedge \frac{Y.}{W'.} \rightarrow 0
\end{aligned}\]}
of complexes is exact as well, as was to be shown.
\end{proof}

For the next definition, we in addition assume that every exact category~$\cM_n$ is equipped with a subcategory $w\cM_n$ of weak equivalences in the sense of \cite[Section~1.2]{Wald} or \cite[Appendix A]{Grayson2012}. We then define the subcategory $wF_k(\cM_n)$ of $F_k(\cM_n)$ by keeping all objects of $F_k(\cM_n)$ but keeping only those morphisms in $F_k(\cM_n)$ which levelwise belong to $w\cM_n$.

\begin{defn}\label{def: preserves weak equivalences}
We say that the {\em assembly of power operations preserves weak equivalences} if the given power  operations
\[(V,W) \mapsto V\otimes W \quad \textrm{ and } \quad (V_1 \rightarrowtail \ldots \rightarrowtail V_k) \mapsto V_1 \wedge \ldots \wedge V_k\]
restrict to well-defined functors
\[w\cM_n \times w\cM_p \rightarrow w\cM_{n+p} \quad \textrm{ and } \quad wF_k(\cM_n) \rightarrow w\cM_{nk}\]
for all $k,n$ and $p$.
\end{defn}

The main purpose of the rest of this section is to construct examples for this definition, see \cref{Prop: DoldPuppeOperations II} and \cref{cor: Dold-Puppe operations IV} below. The crucial input here, in the case of the standard assembly of power operations on $\mathcal{P}(X)$, is proved in \cref{prop: exterior powers of quasi-isom} at the end of this section.

\medskip

The second sentence in \cite[A.9.2]{ThomTro} implies that every idempotent complete exact category $\cM$ supports long exact sequences in the sense of \cite[Definition~1.4]{Grayson2012}. This assures that we have a good notion of quasi\-/isomorphisms for complexes in $C\cM$, see \cite[Definition~2.6]{Grayson2012} and the discussion after \cite[Definition~1.4]{Grayson2012}. In particular, the subcategory $qC\cM$ of $C\cM$ obtained from $C\cM$ by removing all morphisms in $C\cM$ which are not quasi-isomorphisms is a subcategory of weak equivalences, see \cite[Section~2]{Grayson2012}.
Let $C^\chi\cM$ denote the full subcategory of $C\cM$ consisting of complexes $V.$ in $C\cM$ whose Euler characteristic $\chi(V.)$ in $K_0(\cM)$ vanishes, and let $C^q \cM$ denote the full subcategory of $C^\chi\cM$ consisting of acyclic complexes.\footnote{The superscript ${}^q$ here is short hand for `\emph{q}uasi-isomorphic to the zero complex'. The letter $q$ in mathematical expressions in this paper will never denote a variable (such as an integer).}

\medskip

In the following proposition we refer to the assembly of power operations on $C\cM_n$, $n\ge 0$, constructed above, see \cref{prop: DoldPuppe Operations}. In particular, for every morphism $f. \colon V. \rightarrow V'.$ in $C\cM_n$ and for every $k\ge 1$, the induced morphism
$\bigwedge^k(f.) \colon \bigwedge^k (V.) := V. \wedge \ldots \wedge V. \rightarrow V'. \wedge \ldots \wedge V'. =: \bigwedge^k (V'.)$ is defined.

\begin{prop}\label{Prop: DoldPuppeOperations II} Assume that, for every quasi-isomorphism $f.$ in $C\cM_n$ and for every $k \ge 1$, also $\bigwedge^k(f.)$ is a quasi-isomorphism. Then the assembly of power operations on $C\cM_n$, $n \ge 0$, preserves quasi-isomorphisms in the sense of \cref{def: preserves weak equivalences}. Furthermore, restricting to $C^\chi\cM_n$ and to $C^q\cM_n$ yields further assemblies of power operations.
\end{prop}

\begin{proof}
We first verify that, if $V. \rightarrow V'.$ and $W. \rightarrow W'.$ are quasi\hyp{}isomorphisms in $C\cM_n$ and $C\cM_p$, respectively, then so is the induced morphism between the simplicial tensor products $V. \otimes_\Delta W.$ and $V'. \otimes_\Delta W'$. This is well-known and easy to prove for the induced morphism between the total complexes of the ordinary tensor products. Then the Eilenberg-Zilber Theorem \cite[Satz~2.9]{DP} implies that it is also true for the induced map between the simplicial tensor products, as claimed.

Now, let $(V_1. \rightarrowtail \ldots \rightarrowtail V_k.) \longrightarrow (V'_1. \rightarrowtail \ldots \rightarrowtail V'_k.)$ be a morphism in $qF_k(C\cM_n)$, i.e., a levelwise quasi-isomorphism.
By axiom (E5), we have the  $k-1$ short exact sequences of complexes
\[\resizebox{0.83\hsize}{!} {$
 0 \;\rightarrow\; V_1. \wedge \ldots \wedge V_k. \; \rightarrow \; V_1. \wedge \ldots \wedge V_{k-2}. \wedge V_k. \wedge V_k. \; \rightarrow \; V_1. \wedge \ldots \wedge V_{k-2}. \otimes_\Delta \bigwedge^2 \left(\frac{V_k.}{V_{k-1}.}\right) \; \rightarrow \; 0$}\]
\[\resizebox{\hsize}{!} {$0 \;\rightarrow\; V_1. \wedge \ldots \wedge V_{k-2}. \wedge V_k. \wedge V_k. \;\rightarrow \; V_1. \wedge \ldots \wedge V_{k-3}. \wedge V_k. \wedge V_k. \wedge V_k. \;\rightarrow \; V_1. \wedge \ldots \wedge V_{k-3}. \otimes_\Delta \bigwedge^3\left(\frac{V_k.}{V_{k-2}.}\right) \;\rightarrow \; 0$}\]
\[\vdots\]
\[\resizebox{0.52\hsize}{!}{$ 0 \; \rightarrow \; V_1. \wedge V_k. \wedge \ldots \wedge V_k. \; \rightarrow \; \bigwedge^k V_k. \; \rightarrow \; \bigwedge^k \left(\frac{V_k.}{V_1.}\right) \; \rightarrow \; 0,$}\]
and similarly for $V$ replaced with $V'$. By assumption, the induced morphisms
$\bigwedge^k V_k. \rightarrow \bigwedge^k V'_k.$ and $\bigwedge^k\left(\frac{V_k.}{V_1.}\right) \rightarrow \bigwedge^k\left(\frac{V'_k.}{V'_1.}\right)$ are quasi-isomorphisms.

Then the bottom short exact sequence shows that the induced morphism $V_1.\wedge  V_k. \wedge \ldots \wedge V_k.\rightarrow V'_1.\wedge  V'_k. \wedge \ldots \wedge V'_k.$ is a quasi-isomorphism as well. We iterate this conclusion bottom up and finally arrive at the desired conclusion that the induced morphism $V_1. \wedge \ldots \wedge V_k. \rightarrow V'_1. \wedge \ldots \wedge V'_k.$ is a quasi-isomorphism. This finishes the proof of the first statement in \cref{Prop: DoldPuppeOperations II}.\\
This first statement immediately implies that restricting the power operations from $C\cM_n$, $n \ge 0$, to $C^q\cM_n$, $n \ge 0$, yields a well-defined assembly of power operations.

We now prove the analogous statement for $C^\chi \cM_n$, $n\ge 0$. To this end, we'll show that the Euler characteristic commutes with exterior powers and hence that applying $\bigwedge^k$ to any complex in $C^\chi \cM_n$ yields a complex in $C^\chi \cM_{nk}$. (Similar reasoning can be applied to show the easier statement that the simplicial tensor product of complexes in $C^\chi \cM_n$ and $C^\chi \cM_p$ is  in~$C^\chi \cM_{n+p}$.)  Let
\begin{equation}\label{eq: definition of K_0Tilde}
\tilde{K}_0(\cM_n) := \mathrm{coker}(K_0(C^q\cM_n) \rightarrow K_0(C\cM_n))
\end{equation}

and let $K_0(C\cM_n)$, $n \ge 0$, be endowed with the product induced by the simplicial tensor product and with the exterior power operations $\lambda^k$ given by \cref{prop: DoldPuppe Operations}. Note that, in the less general situation considered in \cite[Section~6]{HKT17}, these are the same as the ones constructed there. The map
\[\lambda_t\colon K_0(C\cM_n) \rightarrow 1+  \prod_{k \ge 1} K_0(C\cM_{nk})t^k, \quad x \mapsto \sum_{k \ge 0} \lambda^k(x)t^k,\]
is then a homomorphism, as explained for example in \cite[Section 2]{GraysonExterior}. As acyclic complexes are mapped to acyclic complexes under $\bigwedge^k$, this homomorphism induces a homomorphism
\[\lambda_t\colon \tilde{K}_0(\cM_n) \rightarrow 1+  \prod_{k \ge 1} \tilde{K}_0(\cM_{nk})t^k\]
and in particular power operations $\lambda^k$ for $\tilde{K}_0(\cM_n)$.
By construction, the diagram
\[\xymatrix{\tilde{K}_0(\cM_n) \ar[r]^{\lambda^k} & \tilde{K}_0(\cM_{nk})\\
K_0(\cM_n) \ar[u] \ar[r]^{\lambda^k} & K_0(\cM_{nk}) \ar[u]
}\]
commutes; here, the vertical arrows are given by $V \mapsto V[0]$ where $V[0]$ denotes the complex whose object in degree $0$ is $V$ and whose object in every other degree is $0$. These vertical maps are bijective and their inverses are given by the Euler characteristic (see \cref{Lem: GrothendieckGroupComplexes} below). Hence, the Euler characteristic commutes with the power operations $\lambda^k$. In particular, applying $\bigwedge^k$ to any complex in $C^\chi \cM_n$ yields a complex in $C^\chi \cM_{nk}$. Using the same inductive argument as above, we obtain that $V_1. \wedge \ldots \wedge V_k.$ belongs to $C^\chi \cM_{nk}$ for any sequence $V_1. \rightarrowtail \ldots \rightarrowtail V_k.$ in $F_k(C^\chi \cM_n)$, as was to be proved.
\end{proof}

\begin{lem}\label{Lem: GrothendieckGroupComplexes}
Let $\cM$ be an exact category and let $\tilde{K}_0(\cM)$ be defined as in~\eqref{eq: definition of K_0Tilde}. Then the assignment $V \mapsto V[0]$ induces an isomorphism
\[K_0(\cM) \overset{\sim}{\longrightarrow} \tilde{K}_0(\cM).\]
Its inverse is given by mapping the class of a complex $V. \in C(\cP)$ to its Euler characteristic $\chi(V.):= \sum_{i \ge 0} (-1)^i [V_i] \in K_0(\cM)$.
\end{lem}

This elementary lemma is a variant of the probably more widely known fact that the factor group of $K_0(C\cM)$ modulo the usual shifting relations is isomorphic to $K_0(\cM)$. We include a proof for the reader's convenience.

\begin{proof}
We obviously have $\chi(V[0])= V$ for every $V \in \cM$. It therefore suffices to prove that $ \chi(V.)[0] = [V.]$ in $\tilde{K}_0(\cM)$ for every $V. \in C\cM$. An easy induction using naive truncation shows that
\[[V.] = \sum_{i \ge 0} [V_i[i]] \quad \textrm{ in } \quad K_0(C\cM). \]
Furthermore, for any $V \in \cM$, the standard exact sequence
\[0 \rightarrow V[i] \rightarrow \mathrm{cone}(\mathrm{id}_{V[i]}) \rightarrow V[i+1] \rightarrow 0\]
shows that $[V[i+1]] = -[V[i]]$ in $\tilde{K}_0(\cM)$. Hence we have
\[[V.] = \sum_{i \ge 0} (-1)^i[V_i[0]] = \chi(V.)[0] \quad \textrm{ in } \quad \tilde{K}_0(\cM),\]
as was to be shown.
\end{proof}

Let $\cM$ be an exact category. We recall from \cite{Grayson2012} and \cite{HKT17} that a {\em binary complex} $\VV=(V., d, \tilde{d})$ in~$\cM$ is a graded object~$V.$ in $\cM$ together with two degree~$-1$ maps $d, \tilde{d}\colon V. \rightarrow V.$ such that both $d^2=0$ and $\tilde{d}^2=0$.  If $d=\tilde{d}$,  the complex~$\VV$ is said to be {\em diagonal}. A {\em morphism between binary complexes} is a degree 0 map between the underlying graded objects that is a chain map with respect to both differentials. The obvious definition of short exact sequences turns the category of binary complexes into an exact category. The exact full subcategories consisting of binary complexes $\VV$ in $\cM$ such that both $(V., d)$ and $(V.,\tilde{d})$ belong to $C\cM$ or $C^q\cM$ will be denoted $B\cM$ and $B^q\cM$, respectively.

We continue to assume that we are given an assembly of power operations as in \cref{def: power operations} where every category $\cM_n$ is idempotent complete. The following corollary is a variant of the construction in \cite[Section~4]{HKT17}.

\begin{cor}\label{cor: Dold-Puppe operations III}
By applying the construction above to each of the two differentials in binary complexes and restricting the resulting power operations to $B\cM_n$, $n\ge 0$, we obtain another assembly of power operations. Under the same assumption as in \cref{Prop: DoldPuppeOperations II}, restricting to  $B^q\cM_n$, $n \ge 0$, yields a further assembly of power operations.
\end{cor}

\begin{proof}
This immediately follows from \cref{prop: DoldPuppe Operations} and (the proof of) \cref{Prop: DoldPuppeOperations II}. Note that, for every binary complex $\VV$, the objects in $\bigwedge^k(\VV)$  are well-defined because they don't depend on the differentials $d, \tilde{d}$, as explained in \cite[Lemma~4.2]{HKT17}.
\end{proof}

Finally, we recall that by iterating the definitions of $C$ and $B^q$ above, we obtain the exact category $(B^q)^s \cM_n$ of $s$-dimensional acyclic binary multi-complexes in $\cM_n$ and the exact category $C (B^q)^s \cM_n$ of (bounded and non-negatively supported) complexes of objects in $(B^q)^s \cM_n$; see \cite[Definition~1.1(3) and Remark 1.2]{HKT17} for further explanations of these definitions.

\begin{cor}\label{cor: Dold-Puppe operations IV}
Assuming inductively in $s$ that, for every quasi-iso\-mor\-phism~$f.$ in every category $C(B^q)^s \cM_n$ and for every $k \ge 1$, also $\bigwedge^k(f.)$ is a quasi-isomorphism, we obtain, by iterating the construction above, an assembly of power operations on the categories $(B^q)^s \cM_n$, $n\ge 0$, for every~$s \ge 0$.
\end{cor}

\begin{proof}
The statement for $s=0$ just means that we are given an assembly of power operations on $\cM_n$, $n \ge 0$, as assumed. From \cref{cor: Dold-Puppe operations III} we obtain an assembly of power operations on the categories $B^q\cM_n$, $n \ge 0$, proving the statement for $s=1$. Furthermore, by viewing binary complexes as pairs of ordinary complexes, it follows from \cite[Lemma~3.4(1)]{HKT17} that these categories are idempotent complete again. Hence, the assumption for $s=1$ makes sense and induction on~$s$ proves \cref{cor: Dold-Puppe operations IV}.
\end{proof}

Lastly, let $X$ be a quasi-compact scheme and let $\cP(X)$ or just $\cP$ denote the category of locally free $\cO_X$-modules of finite rank. We consider the standard assembly of power operations on $\cP$ given by the usual tensor product and exterior power operations, as explained after \cref{def: power operations}. When $s=0$ the following proposition means that quasi-isomorpisms between complexes in $C \cP$ are preserved by exterior powers.

\begin{prop}\label{prop: exterior powers of quasi-isom}
The assumption of \cref{cor: Dold-Puppe operations IV} is satisfied for the standard assembly of power operations on $\cP$.
\end{prop}

\begin{proof}
By regarding a binary complex as a pair of two ordinary complexes, it suffices to prove this for $B^q$ replaced with $C^q$. For every open subset~$U$ of $X$, we consider the exact category $\cP(U)$ embedded in the abelian category $\cA(U)$ of quasi-coherent sheaves on $U$, and then, for any $s$, the exact category $(C^q)^s\cP(U)$ embedded in the abelian category $C^s\cA(U)$. Then, in order to show that $\bigwedge^k(f.)$ is a quasi-isomorphism, it suffices to show it is so in $C^s\cA(X)$, and by sheaf-theoretic arguments it then suffices to show that $\bigwedge^k(f.)|_U =\bigwedge^k(f.|_U)$ is a quasi-isomorphisms in $C^s\cA(U)$ for every open affine subset $U$ of $X$. Since the restriction $f.|_U$ is a quasi-isomorphism as well, we may therefore assume that $X$ is affine.\\
Then, every object in $\cP(X)$ is projective in $\cA(X)$ and, by \cite[Exercise~2.2.1]{WeiHom} and \cite[Lemma~3.4(2)]{HKT17},
every object in $(C^q)^s \cP(X)$ is projective in $C^s \cA(X)$. The dual of \cite[Lem\-ma~10.4.6]{WeiHom} now implies that $f.$ is a split surjection in the homotopy category and that the same holds for its (a priori) one-sided inverse. Hence, $f.$ is in fact a homotopy equivalence. (This reasoning is taken from the end of the proof of \cite[Theorem~10.4.8]{WeiHom}). Similarly to the proof of \cite[Proposition~3.1(2)]{HKT17}, then also $\bigwedge^k(f.)$ is a homotopy equivalence since the defining relations of a simplicial homotopy (see \cite[Definition 5.1]{May}) are preserved by any functor, in particular by~$\bigwedge^k$.
\end{proof}

\begin{rem}
As explained in \cite[Example~3.2]{HKT17}, it is not true in general that exterior powers of acyclic complexes of {\em coherent} modules are acyclic again. Similarly, \cref{prop: exterior powers of quasi-isom} may not be true when $B^q$ is replaced with $B$ or $C$.
\end{rem}

\section{Operations on \texorpdfstring{$K$}{K}-Theory Spaces}\label{sec: operations on spaces}
In this section we explain how the techniques of \cite{GraysonExterior}, \cite{GSVW} and \cite{GunnSchw} can be used to turn an assembly of power operations on exact categories with weak equivalences into operations on $K$-theory spaces and $K$-groups. In particular, throughout this section, we assume we are given an assembly of power operations as in \cref{def: power operations}. Later on, we furthermore assume that it preserves given weak equivalences in the sense of \cref{def: preserves weak equivalences}.

\medskip

(When recalling various definitions over the next page or so, we will do so only to the extent they are needed to be able to follow the rest of this section and paper; the remaining details, required for  full definitions, are highly combinatorial and lengthy, and will rather be referenced.)
For each~$n$, let $g\cM_n$ denote the $G$-construction applied to the exact category~$\cM_n$ (denoted~$G\cM_n$ in \cite{GG} and \cite{GraysonExterior}). Recall $g\cM_n$ is a simplicial set where $g\cM_n (A)$ is of the form
\[g\cM_n(A) = \mathrm{Exact}(\Gamma(A), \cM_n)\]
for each $A$ in the category $\Delta$, the category of non-empty finite totally ordered sets. Here, $\Gamma(A)$ is the category of all arrows $j \rightarrow i$ in the partially ordered set $A \sqcup \{L,R\}$ such that $i \in A$ and where $L$ and $R$ are non-comparable elements formally added to the ordered set $A$ so that $L <a$ and $R <a$ for every $a \in A$; furthermore, for any $k \le j \le i \in A \sqcup \{L,R\}$ such that $i, j \in A$, the sequence $(k\rightarrow j) \rightarrow (k \rightarrow i) \rightarrow (j\rightarrow i)$ in $A \sqcup \{L,R\}$ is declared to be exact; then, the set $\mathrm{Exact}(\Gamma(A), \cM_n)$ of exact functors from $\Gamma(A)$ to $\cM_n$ has a well-defined and obvious meaning; for more precise details, see \cite[Section~3]{GraysonExterior}. In this paper however, similarly to the viewpoint taken in \cite{GSVW} and \cite[Section~1]{GunnSchw}, we consider $g\cM_n$ not only as a simplicial set but most of the time actually as a simplicial category and then write $G\cM_n$ for it; in fact, we temporarily view $G\cM_n$ as a simplicial object in the category $\cE x$ of exact categories (where the morphisms are {\em exact} functors). All the additional structure is of course derived from the corresponding given structure on $\cM_n$. For example, a sequence
\[M' \rightarrow M \rightarrow M'' \quad \textrm{ in } \quad \mathrm{Exact}(\Gamma(A), \cM_n)\]
is said to be a short exact sequence if $M'(x) \rightarrowtail M(x) \twoheadrightarrow M''(x)$ is a short exact sequence in $\cM_n$ for all $x \in \Gamma(A)$. This slightly more sophisticated viewpoint is already convenient when defining the $k$-fold iterated $G$-construction $G^k \cM_n$. More precisely, the $k$-simplicial object $G^k \cM_n$ in $\cE x$ is given by
\begin{eqnarray*}
\lefteqn{(G^k \cM_n)(A_1, \ldots, A_n)}\\
& := &\mathrm{Exact}(\Gamma (A_n), \mathrm{Exact}(\Gamma (A_{n-1}), \ldots, \mathrm{Exact}(\Gamma (A_1), \cM_n) \ldots))\\
&= & \mathrm{Exact}(\Gamma(A_1) \times \ldots \times \Gamma(A_n), \cM_n)
\end{eqnarray*}
for $A_1, \ldots, A_n \in \Delta$; the last ``Exact'' here means ``multi-exact'', similarly to ``multi-linear'' in Linear Algebra. We recall furthermore from \cite[Section~4]{GraysonExterior} that composing with concatenation
\[\Delta^k \rightarrow \Delta, \quad (A_1, \ldots, A_k) \mapsto A_1 \cdots A_k,\]
turns every simplicial object $X$ into a $k$-simplicial object $\mathrm{Sub}_k X$, called the \emph{$k$-fold edgewise subdivision of $X$}.
Finally, for each $k$ and $A \in \Delta$, Grayson \cite[Sections~5,~6]{GraysonExterior} defines the set $\Gamma^k(A)$ consisting of (formal) tuples $(x_1, \ast_2, x_2, \ast_3, \ldots, \ast_k, x_k)$ where $x_r \in \Gamma(A)$ and $\ast_r \in \{\otimes, \wedge\}$ for each $r$ and which satisfy certain conditions. He furthermore equips $\Gamma^k(A)$ with a certain partial order and a notion of short exact sequences and defines a particular multi-exact map
\[\Xi\colon \Gamma(A_1) \times \ldots \times \Gamma(A_k) \rightarrow \Gamma^k(A_1 \cdots A_k)\]
for any $A_1, \ldots, A_k \in \Delta$.

\begin{defn}
For any $k \ge 1$, $n \ge 0$ and $A_1, \ldots, A_n \in \Delta$, let
\[\lambda^k(A_1, \ldots, A_k) \colon \mathrm{Sub}_k (G\cM_n)(A_1, \ldots, A_n) \rightarrow G^k\cM_{nk}(A_1, \ldots, A_k)\]
denote the composition of the functors
\begin{align*}
\mathrm{Exact}(\Gamma(A), \cM_n) &&&\rightarrow &&\mathrm{Exact}(\Gamma^k(A), \cM_{nk})\\
M &&&\mapsto &&\left((x_1, \ast_2, \ldots \ast_k, x_k) \mapsto  M(x_1) \ast_2 \ldots \ast_k M(x_k)\right)
\end{align*}
(where $A:= A_1 \cdots A_k$) and the functor
$\mathrm{Exact}(\Xi, \cM_{nk})$ between the categories $\mathrm{Exact}(\Gamma^k(A_1 \cdots A_k), \cM_{nk})$ and $\mathrm{Exact}(\Gamma(A_1)\times \ldots \times \Gamma(A_k), \cM_{nk})$.
\end{defn}

\begin{prop} \label{prop: lambda operation as k-simplicial map}
The previous definition yields a well-defined morphism
\[\lambda^k\colon \mathrm{Sub}_k (G\cM_n) \rightarrow G^k\cM_{nk}\]
in the category of $k$-simplicial categories.
\end{prop}

\begin{proof}
In \cite[Section~7]{GraysonExterior}, this morphism is shown to be well-defined in the category of $k$-simplicial sets. The slightly stronger statement of this proposition then immediately follows from the functoriality of the morphisms assumed in (E1, E2) of \cref{def: power operations}.
This stronger statement is actually also stated in a more general context in \cite[Theorem~4.1]{GunnSchw}, but the type of simplicial objects considered there is not clearly specified: beware this is not a map between $k$-simplicial objects in $\cE x$.
\end{proof}

The object of the rest of this section is to pass from the operation $\lambda^k$ defined above to a map between $K$-theory spaces and finally between $K$-groups. To this end, we in addition assume from now on that every category $\cM_n$ is equipped with a subcategory $w\cM_n$ of weak equivalences and that the given assembly of power operations preserves weak equivalences.

\medskip

Similarly to the definition of $wF_k(\cM_n)$ in \cref{Sec: Operations}, let the $k$-simplicial category $wG^k\cM_n$ be defined by the requirement that $wG^k\cM_n(A_1, \ldots, A_k)$ is the subcategory of $G^k\cM_n(A_1, \ldots, A_k)$ whose morphisms levelwise belong to $w\cM_n$ (for any $A_1, \ldots, A_k \in \Delta$).

It is easy to see that, under the additional assumption above, the operation $\lambda^k \colon \textrm{Sub}_kG\cM_n \rightarrow G^k \cM_n$ from \cref{prop: lambda operation as k-simplicial map} induces a well-defined morphism
\begin{equation}\label{Eq: morphism w lambda}
w\lambda^k \colon \textrm{Sub}_k wG\cM_n \rightarrow wG^k \cM_{nk}
\end{equation}
between $k$-simplicial categories (for every $k$). Let $N$ denote the functor from the category of categories to the category of simplicial sets which maps a category to its nerve. By composing with $N$ we then obtain a map
\[N \circ \textrm{Sub}_k wG\cM_n \rightarrow N \circ w G^k\cM_{nk}\]
between $k$-simplicial objects in the category of simplicial sets, i.~e., a map between $(k+1)$-simplicial sets. Let
\[|w \lambda^k| \colon |\textrm{Sub}_k wG \cM_n| \rightarrow | wG^k\cM_{nk} | \]
denote the continuous map obtained from the previous map by passing to geometric realisations, see \cite[Section~2]{Petric} for details.

The following theorem will finally allow us to define the operation $\lambda^k$ on $K$-theory spaces and $K$-groups. Note that, for the construction of $|wG^k \cM_n|$ and $|\textrm{Sub}_k w G\cM_n|$, the exact category $\cM_n$ and its subcategory $w\cM_n$ of weak equivalences need not be part of an assembly of power operations.

\begin{thm} \label{Thm: Homotopy Equivalences}
Let $\cM$ be an exact category equipped with a subcategory $w\cM$ of weak equivalences. Then: \\
(a) We have a natural homeomorphism $ |wG \cM|\rightarrow |\mathrm{Sub}_k wG\cM|$.\\
(b) If $w\cM$ satisfies the extension axiom (see \cite[Section~1.2]{Wald}), we have a natural homotopy equivalence $|wG\cM| \rightarrow |wG^k\cM|$.
\end{thm}

\begin{proof} $ $ \newline
(a) \cite[Section~4]{GraysonExterior} provides a natural homeomorphism $ |X|\rightarrow |\mathrm{Sub}_k X|$ for any simplicial set $X$. It is easy to verify that this map remains a homeomorphism if $X$ is a simplicial topological space. We therefore obtain the desired homeomorphism $ |wG \cM|\rightarrow |\mathrm{Sub}_k wG\cM|$ by applying this to the simplicial topological space which maps $[k]$ to the classifying space $BwG^k\cM$ of the category $wG^k\cM$. Note that realising a multisimplicial set does not depend on the order of simplicial directions. \\
(b) We begin with $k=2$. By \cite[\href{https://kerodon.net/tag/00FV}{Tag 00FV}]{kerodon}, the functor $\cM\rightarrow \cM\times\cM\cong G\cM([0])$ given by  $V\mapsto (V,0)$ induces a morphism of simplicial categories $\cM\rightarrow G\cM$ where we view the left-hand side as a constant simplicial category. We apply $wG$ and denote the resulting morphism by $\varphi \colon wG\cM\rightarrow wG^2\cM$. Its realisation is shown to be a homotopy equivalence in part (a) of the proof of \cite[Theorem~2.6]{GSVW}. Note that the $G$-construction there when applied to an exact category with weak equivalences is indeed the same as the construction described in this section.

This proof can be generalised to $k\ge 2$ as follows. Arguing as in \cite[\href{https://kerodon.net/tag/00FV}{Tag 00FV}]{kerodon},
the functor
\[\cM\rightarrow \cM^{2^{k-1}}\cong G^{k-1}\cM(([0], \ldots, [0])), \quad V \mapsto (V,0,...,0)\]
induces a morphism of ($k-1$)-simplicial categories $\tau \colon \cM \rightarrow G^{k-1}\cM$ where we view $\cM$ as a constant ($k-1$)-simplicial category. More precisely, for any $A_1,...,A_{k-1}\in\Delta$, the functor $\tau((A_1, \ldots, A_{k-1}))$ is the composition of the functor $V \mapsto (V,0,...,0)$ with the functor $\cM^{2^{k-1}}\rightarrow G^{k-1}\cM(A_1,...,A_{k-1})$
associated by $G^{k-1}\cM$ to the unique map $(A_1,...,A_{k-1})\rightarrow ([0], \ldots, [0])$ in~$\Delta^{k-1}$. Finally, we apply $wG$, pass to realisations and show as in \cite{GSVW} that the resulting map is a homotopy equivalence.
\end{proof}

Using the previous theorem, we obtain the sequence of continuous maps
\begin{equation}\label{Eq: morphism lambda}
|wG\cM_n| \xrightarrow{\cong} |\mathrm{Sub}_k wG\cM_n| \xrightarrow{|w\lambda^k|} |wG^k \cM_{nk}| \xleftarrow{\simeq} |wG\cM_{nk}|
\end{equation}
(for any $n \ge 0$ and $k\ge 1$) whose composition, after inverting the right-hand homotopy equivalence, we view as a well-defined morphism from $|wG\cM_n|$ to $|wG\cM_{nk}|$ in the (naive) homotopy category of topological spaces, and which in particular then does not depend on the choice of a particular continuous map representing the right-hand inverse. By abuse of notation, we denote this composition~$\lambda^k$ again.

\medskip

Recall, by composing with the forgetful functor (forgetting the morphisms) from the category of categories to the category of sets, we can associate a $k$-simplicial set $x$ with any $k$-simplicial category $X$. In other words, $x$ is the $k$-simplicial set of $0$-simplices of the simplicial object $N\circ X$. By \cite[\href{https://kerodon.net/tag/00FV}{Tag 00FV}]{kerodon}, we therefore have a morphism of ($k+1$)-simplicial sets $x\rightarrow N\circ X$. Denoting the $k$-simplicial set corresponding to the $k$-simplicial category~$G^k\cM_n$ by $g^k\cM_n$, we hence get morphisms $g^k\cM_n\rightarrow N \circ wG^k\cM_n$ (note that for any multi-index $A_1, \ldots, A_k \in \Delta$, the categories $wG^k\cM_n(A_1, \ldots A_k)$ and $G^k\cM_n(A_1, \ldots, A_k)$ have the same objects). Therefore we get maps $|g^k \cM_n | \rightarrow |wG^k\cM_n|$ and, similarly, maps $|\mathrm{Sub}_k g\cM_n| \rightarrow |\mathrm{Sub}_k wG\cM_n|$.

\begin{prop}\label{Prop: Step 0}
For any $n\ge 0$ and $k\ge 1$, we have the commutative diagram
\begin{equation}\label{Eq: comparing g and G}
\xymatrix{|g\cM_n| \ar[d] \ar[r]& |\mathrm{Sub}_k g\cM_n| \ar[r]^{|\lambda^k|} \ar[d] & |g^k\cM_{nk}| \ar[d] & |g\cM_{nk}| \ar[l] \ar[d]\\
|wG\cM_n| \ar[r] & |\mathrm{Sub}_k wG\cM_n| \ar[r]^{|w\lambda^k|} & |wG^k\cM_{nk}| & |wG\cM_{nk}| \ar[l]}
\end{equation}
where the spaces and the maps in the top line are the ones defined in \cite[Sections~4-7]{GraysonExterior} and the bottom line is sequence \eqref{Eq: morphism lambda}. If moreover $w\cM_n$, for every $n$, is the category of weak equivalences consisting only of isomorphisms (also denoted $i\cM_n$), then all vertical maps are homotopy equivalences.
\end{prop}

\begin{proof}
The commutativity of the left-hand square follows from the explanations in the proof of part (a) in \cref{Thm: Homotopy Equivalences}. It follows immediately from the definition of the vertical maps and from the details provided above for the construction of the maps~$|\lambda^k|$ and $|w\lambda^k|$ that the middle square commutes.  The commutativity of the right-hand square is a consequence of the description of the horizontal maps given in \cite[Lemma~6.3]{GG} and in the proof of part~(b) in \cref{Thm: Homotopy Equivalences}.
If $w\cM_n = i\cM_n$, the left- and right-hand vertical maps are homotopy equivalences by \cite[Lem\-ma~2.14]{GSVW}. By \cref{Thm: Homotopy Equivalences} and the analogous statements in \cite{GraysonExterior}, then also the two middle vertical maps are homotopy equivalences.
\end{proof}

Recall that \cite[Theorem~3.1]{GG} provides a natural isomorphism between the $r^\mathrm{th}$ homotopy group $\pi_r(|g\cM|)$ and Quillen's $K$-group $K_r(\cM)$ for any exact category $\cM$ and $r \ge 0$. By passing to homotopy groups in the top row of diagram (\ref{Eq: comparing g and G}), we thus finally obtain the operation
\begin{equation}\label{Eq: Old lambda operation}
\lambda^k: K_r(\cM_n) \rightarrow K_r(\cM_{nk}),
\end{equation}
the main outcome of \cite{GraysonExterior}. Via the commutative diagram (\ref{Eq: comparing g and G}), we may identify this operation with the map obtained by passing to homotopy groups of the map $\lambda^k \colon |iG\cM_n| \rightarrow |iG\cM_{nk}|$ given by \eqref{Eq: morphism lambda}. Furthermore, when $r=0$, this operation agrees with the standard and classical operation $\lambda^k$ on Grothendieck groups, see \cite[Section~8]{GraysonExterior}. Finally, we recall \cite[Section~9]{GraysonExterior}, when $r \ge 1$ and the given assembly of power operations is given by exterior powers of finitely generated projective modules over a commutative ring $R$, this operation agrees with the $\lambda$-operation $\lambda^k \colon K_r(R) \rightarrow K_r(R)$ defined in \cite{Hiller} and \cite{Kratzer}.

\medskip

We conclude this section with variants of Waldhausen fibration theorem which will be crucial for the next section.

\begin{thm}\label{Thm: WaldFib}
Let $\cM$ be an exact category and let $v\cM\subset w\cM$ be two subcategories of weak equivalences. Suppose in addition that $(\cM,w)$ has a cylinder functor satisfying the cylinder axiom (see \cite[p.~348--349]{Wald}). Let $\cM^w$ denote the full subcategory of $\cM$ consisting of $w$-acyclic objects. Then, for any $k\geq 1$, we have the following two homotopy fibre sequences:
\begin{equation}\label{Eq: G^k fibre sequence}
|vG^k\cM^w|\rightarrow |vG^k\cM|\rightarrow |w G^k\cM|
\end{equation}
\begin{equation}\label{Eq: Sub_k fibre sequence}
|\Sub_kvG\cM^w|\rightarrow |\Sub_kvG\cM|\rightarrow|\Sub_k w G\cM|
\end{equation}
\end{thm}
\begin{proof}

According to the original Waldhausen fibration theorem (see \cite[1.6.4]{Wald} or \cite[1.8.2]{ThomTro}) we have the homotopy fibre sequence
\begin{equation}\label{Eq: S. fibre sequence}
|vS.\cM^w|\rightarrow |vS.\cM|\rightarrow |wS.\cM|.
\end{equation}
Stating that the sequences \eqref{Eq: G^k fibre sequence}, \eqref{Eq: Sub_k fibre sequence} and \eqref{Eq: S. fibre sequence} are fibre sequences assumes that they come with a natural null-homotopy; these null-homotopies arise from the fact that the spaces $|wG^k\cM^w|$, $|\mathrm{Sub}_kwG\cM^w|$ and $|wS.\cM^w|$ are contractible which in turn follows from the fact that the distinguished object $0$ (the `diagram' consisting of zeroes) at each level of the ($k$-)simplicial categories $wG^k\cM$, $wS.\cM$ and $\mathrm{Sub}_kwG\cM$ is an initial object in the subcategories $wG^k\cM^w$, $wS.\cM^w$ and $\mathrm{Sub}_kwG\cM^w$. \\
An exact category is obviously pseudo-additive in the sense of \cite[Definition 2.3]{GSVW}. By \cite[Theorem~2.6]{GSVW} we therefore have a homotopy equivalence between $\Omega |wS.\cM|$ and $|wG\cM|$ and similarly between $\Omega |vS.\cM|$ and $|vG\cM|$ and between $\Omega |vS.\cM^w|$ and $|vG\cM^w|$.
These homotopy equivalences define a morphism between the sequence \eqref{Eq: S. fibre sequence} and the sequence \eqref{Eq: G^k fibre sequence} for $k=1$ which is in fact a morphism between sequences with null homotopies. Similarly, \cref{Thm: Homotopy Equivalences} defines a morphism between the latter sequence and the sequence \eqref{Eq: G^k fibre sequence} or \eqref{Eq: Sub_k fibre sequence} which again is a morphism between sequences with null homotopies. (In order to verify the latter for example for \eqref{Eq: Sub_k fibre sequence}, we use the fact that the homeomorphism in \cref{Thm: Homotopy Equivalences}(a) is natural, as stated there, and  compatible with products, after interpreting a natural transformation between functors from a category~$\cC$ to a category~$\mathcal{D}$ as a functor from $\cC \times \{0 < 1\}$ to $\mathcal{D}$.) The sequences in the statement are then homotopy fibre sequences by \cref{Cor: eq of sequences}.
\end{proof}

\section{Combinatorially versus Homotopically Defined Operations}\label{sec: comparing lambda operations}

The goal of this section is to prove that the exterior power operations constructed in \cite{HKT17} combinatorially on higher $K$-groups of a scheme agree with those defined in \cite{GraysonExterior} using homotopical methods. We will do this in the general setup introduced in \cref{Sec: Operations}. In particular, we continue to assume that we are given an assembly of power operations as in \cref{def: power operations}. We furthermore assume:
\begin{itemize}
\item Every exact category $\cM_n$ is idempotent complete.

\item For any $s \ge 0$, $n \ge 0$ and $k \ge 1$ and for every quasi-morphism $f.$ in any $C(B^q)^s \cM_n$, also $\bigwedge^k(f.)$ is a quasi-isomorphism.

\end{itemize}

Note that, as in \cref{cor: Dold-Puppe operations IV}, the assumption for $s \ge 1$ in the second bullet point above makes sense only after the same condition for $s-1$ is assumed to be true. To ease any potential worries about the combinatorial complexity of this section, we remark already now that the assumption for $s \ge 1$ comes into play only at the very end of this section when we apply the inductive hypothesis in the proof of our main theorem (\cref{Thm: main thm}) to the induced assembly of power operations on the categories $B^q\cM_n$, $n \ge 0$. In particular, throughout this section (apart from the very end), we will consider only one-dimensional complexes and we will use the second bullet point only when~$s=0$.

Recall the main outcome of \cite{Grayson2012} is an algebraic description of Quillen's $K$-groups and this identification is achieved via a number of isomorphisms. In order to prove our main theorem (\cref{Thm: main thm}), we will successively verify in the statements below that each of these isomorphisms is compatible with (exterior) power operations. Note however that we use spaces while Grayson uses spectra. The relevant commutative diagrams are presented in a way so they can be joined to one another.

To avoid repetitions, when we say below that an assembly of power operations preserves some weak equivalences, we tacitly assume that the symbol~$\lambda^k$ denotes the corresponding morphism $\lambda^k$ introduced in the previous section, see \eqref{Eq: morphism lambda} and \eqref{Eq: Old lambda operation}.

For any idempotent complete exact category $\cM$, let $i\cM$ denote the category of weak equivalences for $\cM$ consisting only of isomorphisms in $\cM$, let $qC\cM$ denote the category of weak equivalences for $C\cM$ consisting of quasi-isomorphisms in $C\cM$ and let $qB\cM$ denote the category of weak equivalences for $B\cM$ consisting of morphisms which are quasi-isomorphisms with respect to both differentials. The standard inclusion functor $\cM \rightarrow C\cM$, $V \mapsto V[0]$, then induces a continuous map
\[|iG\cM| \rightarrow |qGC\cM|\]
which is a homotopy equivalence by \cite[Theorem~1.11.7]{ThomTro}.
It is obvious that every assembly of power operations preserves isomorphisms. We consider the categories $C\cM_n$ to be equipped with the assembly of power operations according to \cref{prop: DoldPuppe Operations} and \cref{Prop: DoldPuppeOperations II}, which by assumption preserves quasi-isomorphisms.

\begin{lem}\label{lem: first step}
For any $n\ge 0$ and $k\ge 1$, the following diagram commutes up to homotopy:
\[\xymatrix{
|iG\cM_n| \ar[r]^{\lambda^k} \ar[d] & |iG\cM_{nk}| \ar[d]\\
|qGC\cM_n| \ar[r]^{\lambda^k} & |qGC\cM_{nk}|
}\]
\end{lem}

\begin{proof}
This immediately follows from the construction of the maps $\lambda^k$ in the previous section.
\end{proof}

Similarly to $qC\cM_n$, let $qC^\chi \cM_n$ denote the category of weak equivalences for the subcategory $C^\chi\cM_n$ of $C\cM_n$ consisting of quasi-isomorphisms in $C^\chi\cM$. The inclusion $C^\chi \cM_n \subseteq C \cM_n$ then induces a continuous map
\[|qGC^\chi \cM_n| \rightarrow |qGC\cM_n|\]
which after applying $\pi_r$ for $r\ge 1$ becomes an isomorphism by \cite[Remark~5.8]{Grayson2012}. Furthermore, we consider the categories $C^\chi \cM_n$ to be equipped with the assembly of power operations according to \cref{Prop: DoldPuppeOperations II}. It again preserves quasi-isomorphisms.

\begin{lem}\label{lem: second step}
For any $n\ge 0$ and $k\ge 1$, the following diagram commutes up to homotopy:
\[\xymatrix{
|qGC\cM_n| \ar[r]^{\lambda^k} & |qGC\cM_{nk}|\\
|qGC^\chi \cM_n| \ar[r]^{\lambda^k}\ar[u] & |qGC^\chi \cM_{nk}| \ar[u]
}\]
\end{lem}

\begin{proof}
Again, this immediately follows from the construction of $\lambda^k$.
\end{proof}

For $n \ge 0$, let $bB\cM_n$ denote the category of weak equivalences for $B\cM_n$ consisting of morphisms $f$ in $B\cM_n$ which become quasi-isomorphisms in $C\cM_n$ after applying the functor $\bot_n$, the functor which forgets the top differential in $B\cM_n$. Furthermore, let $B^b \cM_n$ denote the full subcategory of $B\cM_n$ consisting of binary complexes which become acyclic after applying the functor~$\bot_n$. Finally, let $qB^b\cM_n$ denote the category of weak equivalences for $B^b\cM_n$ consisting of morphisms which are quasi-isomorphisms with respect to the top differential. Similarly to \cref{Prop: DoldPuppeOperations II} and \cref{cor: Dold-Puppe operations III}, we obtain an assembly of power operations on the categories $B^b \cM_n$, $n\ge 0$, which preserves these weak equivalences.
The analogously defined functor~$\top_n$ (forgetting the bottom differential) induces a continuous map
\[\top_n \colon |qGB^b\cM_n| \rightarrow |qGC^\chi \cM_n|\]
which is a homotopy equivalence by \cite[Theorem~5.9]{Grayson2012}.

\begin{lem}\label{lem: third step}
For any $n\ge 0$ and $k\ge 1$, the following diagram commutes up to homotopy:
\[\xymatrix{
|qGC^\chi\cM_n| \ar[r]^{\lambda^k} & |qGC^\chi\cM_{nk}|\\
|qGB^b \cM_n| \ar[r]^{\lambda^k}\ar[u]_{\top_n} & |qGB^b \cM_{nk}| \ar[u]_{\top_{nk}}
}\]
\end{lem}

\begin{proof}
Again, this immediately follows from the construction of $\lambda^k$.
\end{proof}

(From now on, we use some largely standard notations and terminology about homotopy fibres and fibre sequences introduced in \cref{Sec: Homotopy}.) Recall that, by Waldhausen's fibration theorem (\cref{Thm: WaldFib}), for every $n$,
the obvious sequence
\begin{equation}\label{Eq: fibre sequence}
|qGB^b\cM_n| \rightarrow |qGB\cM_n| \xrightarrow{\varepsilon_n} |bGB\cM_n|
\end{equation}
of continuous maps is a fibre sequence (where we have introduced the notation $\varepsilon_n$ for the right-hand map), i.e., the corresponding whisker map
\[|qGB^b\cM_n| \ra \mathrm{hofib}(\varepsilon_n)\]
(see \cref{def: null-homotopy}) is a homotopy equivalence. The sequences \eqref{Eq: morphism lambda} of continuous maps defining the operations \[\lambda^k \colon |qGB\cM_n| \rightarrow |qGB\cM_{nk}| \quad \textrm{and} \quad \lambda^k \colon |bGB\cM_n| \rightarrow |bGB\cM_{nk}|\]
then induce a  sequence of continuous maps defining an operation
\[\lambda^k \colon \mathrm{hofib}(\varepsilon_n) \rightarrow \mathrm{hofib}(\varepsilon_{nk}).\]

\begin{prop}\label{Prop: identifying lambda operations}
For any $n \ge 0$ and $k \ge 1$, the following diagram commutes up to homotopy:
\[\xymatrix{
|qGB^b\cM_n| \ar[r]^{\lambda^k} \ar[d] & |qGB^b\cM_{nk}| \ar[d] \\
\mathrm{hofib}(\varepsilon_n) \ar[r]^{\lambda^k} & \mathrm{hofib}(\varepsilon_{nk})
}\]
\end{prop}

\begin{proof}
According to \cref{lem: Fiber compatibility}, it suffices to show that each square in the following diagram commutes:
\[\xymatrix{
|qGB^b\cM_n| \ar[r] \ar[d] & |\mathrm{Sub}_k qGB^b\cM_n| \ar[r] \ar[d]& |qG^kB^b\cM_{nk}| \ar[d]& |qGB^b\cM_{nk}| \ar[d]\ar[l]\\
P|bGB\cM_n| \ar[r] & P|\mathrm{Sub}_k bGB\cM_n| \ar[r]& P|bG^kB\cM_{nk}| & P|bGB\cM_{nk}| \ar[l]
}\]
Here, the horizontal sequences are induced by the sequence \eqref{Eq: morphism lambda} and the vertical morphisms are the null homotopies introduced in the proof of \cref{Thm: WaldFib}. The commutativity of the left- and right-hand square has already been observed there. The commutativity of the middle square
follows just from the fact that the morphism \eqref{Eq: morphism w lambda} obviously maps $0$ to $0$ at each level.
\end{proof}

We now consider the fibre sequence
\[\xymatrix{\mathrm{hofib}(\bot_n) \ar[r] &  |qGB\cM_n| \ar[r]^{\bot_n} & |qGC\cM_n|}\]
and let
\[\lambda^k \colon \mathrm{hofib}(\bot_n) \rightarrow \mathrm{hofib}(\bot_{nk})\]
denote the operation induced by the operations
\[\lambda^k \colon |qGB\cM_n| \rightarrow |qGB\cM_{nk}| \quad \textrm{and} \quad \lambda^k \colon |qGC\cM_n| \rightarrow |qGC\cM_{nk}|.\]

\begin{lem}\label{lem: fifth step}
For any $n\ge 0$ and $k\ge 1$, we have the following natural homotopy commutative diagram:
\[\xymatrix{
\mathrm{hofib}(\varepsilon_n) \ar[r]^{\lambda^k} \ar[d] & \mathrm{hofib}(\varepsilon_{nk}) \ar[d] \\
\mathrm{hofib}(\bot_n) \ar[r]^{\lambda^k} & \mathrm{hofib}(\bot_{nk})
}\]
\end{lem}

Note that again the vertical maps are homotopy equivalences; this follows from \cite[Theorem~4.8]{Grayson2012} in conjunction with \cref{prop: bounded vs nonnegative complexes} and \cite[Theorem~2.6]{GSVW} (see also proof of \cref{Thm: WaldFib}).

\begin{proof}
This follows from the functoriality of the homotopy fibre applied to the following commutative diagram:
\[\xymatrix{
|qGB\cM_n| \ar[r]^{\varepsilon_n} \ar[d]^{\mathrm{id}} & |bGB\cM_n| \ar[d]^{\bot_n}\\
|qGB\cM_n| \ar[r]^{\bot_{n}} & |qGC\cM_n|
}\]
\end{proof}

From now on, we simplify the proof of Grayson's explicit description of higher $K$-groups \cite[Corollary~7.4]{Grayson2012} inasmuch as we don't use homotopy cofibres. Those steps in his proof which {\em don't} involve homotopy cofibres (and also the proof of \cite[Theorem 1.11.7]{ThomTro} cited later in the proof of \cref{Thm: main thm}) can be most easily translated from the language of spectra to the language of spaces. However, this seems not to be possible for some of the remaining steps (such as \cite[Corollary~7.1]{Grayson2012}) because homotopy cofibres between spaces behave differently to homotopy cofibres between spectra. More importantly, lambda operations seem not to exist on spectra and, a fortiori, not on the homotopy cofibres involved. Luckily, we can do the remaining steps in the more elementary context of homotopy groups. This way, we ultimately completely avoid the use of spectra in the proof of \cite[Corollary~7.4]{Grayson2012}, rendering this proof somewhat more elementary.

\medskip

We consider the continuous map $\Delta_n \colon |qGC\cM_n| \rightarrow |qGB\cM_n|$ induced by the functor $(V.,d) \mapsto (V.,d,d)$.
Similarly to above, let
\[\lambda^k \colon \mathrm{hofib}(\Delta_n) \rightarrow \mathrm{hofib}(\Delta_{nk})\]
denote the induced operation.

\begin{lem}\label{lem: sixth step}
Let $n\ge 0, k\ge 1$ and $r \ge 1$. The following two squares of induced homomorphisms between homotopy groups commute:
\[\xymatrix{
\pi_r(\mathrm{hofib}(\bot_n)) \ar[r]^{\lambda^k} \ar[d]^{\mathrm{can}} & \pi_r(\mathrm{hofib}(\bot_{nk})) \ar[d]^{\mathrm{can}} \\
\pi_r(|qGB\cM_n|) \ar[r]^{\lambda^k} \ar[d]^{\partial} & \pi_r(|qGB\cM_{nk}|) \ar[d]^{\partial}\\
\pi_{r-1}(\mathrm{hofib}(\Delta_n)) \ar[r]^{\lambda^k} & \pi_{r-1}(\mathrm{hofib}(\Delta_{nk}))
}\]
\end{lem}

\begin{proof}
This follows from the functoriality of the long exact sequence of homotopy groups.
\end{proof}

By induction on $r$, we define $K_r^\mathrm{Gr}(\cM_n)$ to be the cokernel of the induced split injective homomorphism $\Delta_n \colon K_{r-1}^\mathrm{Gr}(C^q\cM_n) \rightarrow K_{r-1}^\mathrm{Gr}(B^q\cM_n)$ with $K_0^\mathrm{Gr}$ defined to be the classical Grothendieck group functor $K_0$. This amounts to the description of Quillen's $r^\mathrm{th}$ $K$-group $K_r(\cM_n)$ given in \cite[Corollary~7.4]{Grayson2012}. Note that, contrary to \cite{Grayson2012}, categories of complexes such as $C^q\cM_n$ and $B^q\cM_n$ are defined in this paper using complexes which are supported only in non-negative degrees. However, \cite[Proposition~1.4]{HKT17} states that this doesn't entail any difference for $K_r^\mathrm{Gr}(\cM_n)$, see also \cref{Sec: non-negative complexes}. The induced assemblies of operations on the categories $C\cM_n$ and $B\cM_n$ already considered above inductively induce a map
\begin{equation}\label{Eq: New lambda operation}
\lambda^k \colon K_{r}^\mathrm{Gr}(\cM_n) \rightarrow K_{r}^\mathrm{Gr}(\cM_{nk})
\end{equation}
for every $n\ge 0$, $k\ge 1$ and $r \ge 0$, which is a homomorphism when $r\ge 1$. To see this we use \cite[Theorem~6.2]{HKT17} which also applies to the more general situation when the standard assembly of power operations on the category $\cP(X)$ there is replaced with an arbitrary assembly of power operations. The required short exact sequences~(6.4) for the base step in the inductive proof there indeed follow from the properties of an assembly of power operations (see \cref{def: power operations}).

\medskip

We are now ready to state and prove our main theorem.

\begin{thm}\label{Thm: main thm}
For any $n\ge 0, k\ge 1$ and $r \ge 0$, the following diagram of maps commutes:
\[\xymatrix{
K_r(\cM_n) \ar[r]^{\lambda^k} \ar[d]^\cong & K_r(\cM_{nk}) \ar[d]^\cong\\
K_r^\mathrm{Gr}(\cM_n) \ar[r]^{\lambda^k} & K_r^\mathrm{Gr}(\cM_{nk})
}\]
Here, the vertical bijections are given by \cite[Corollary~7.4]{Grayson2012} and the top and bottom operations are given by \eqref{Eq: Old lambda operation} and \eqref{Eq: New lambda operation}, respectively.
\end{thm}

\begin{proof}
We proceed by induction on $r$. The statement for $r=0$ is true by definition.

For the inductive step $r-1 \rightarrow r$ where $r \ge 1$, we first note that passing to the $r^\mathrm{th}$ homotopy group turns the vertical maps in \cref{lem: first step}, \cref{lem: second step}, \cref{lem: third step}, \cref{Prop: identifying lambda operations} and \cref{lem: fifth step}  into isomorphisms, as stated there. Furthermore, a diagram chase as in the proof of \cite[Corollary~4.5]{Grayson2012} shows that the vertical composition in \cref{lem: sixth step} is an isomorphism. Pre-composing the composition of all these isomorphisms with the isomorphism induced by the vertical homotopy equivalence from \cref{Prop: Step 0}
and post-composing with the isomorphism \eqref{Eq: final isom} below (introducing which is the object of the next paragraph) then yields the vertical isomorphism in \cref{Thm: main thm}, see \cite{Grayson2012}.  By \cref{Prop: Step 0}, \cref{lem: first step}, \cref{lem: second step}, \cref{lem: third step}, \cref{Prop: identifying lambda operations}, \cref{lem: fifth step} and \cref{lem: sixth step}, it therefore remains to show that the isomorphism \eqref{Eq: final isom} below is compatible with $\lambda^k$ (which will be done in the paragraph right after \eqref{Eq: final isom}).

Another application of Waldhausen's fibration theorem (see \cref{Thm: WaldFib}) yields that the horizontal sequences in the following commutative diagram of continuous maps are fibre sequences:
\[\xymatrix{
|iGC^q\cM_n| \ar[r] \ar[d]^{\Delta_n} & |iGC\cM_n| \ar[r] \ar[d]^{\Delta_n} & |qGC\cM_n| \ar[d]^{\Delta_n}\\
|iGB^q\cM_n| \ar[r] & |iGB\cM_n| \ar[r] & |qGB\cM_n|
}\]
Here, all vertical maps are induced by the functor $(V.,d) \mapsto (V.,d,d)$. This diagram induces the following commutative diagram of homomorphisms between homotopy groups, with exact rows and split-exact columns:
\[\xymatrix{
\pi_r(|qGC\cM_n|) \ar[r]^{\partial} \ar@{>->}[d]^{\Delta_n} & \pi_{r-1}(|iGC^q\cM_n|) \ar[r] \ar@{>->}[d]^{\Delta_n} & \pi_{r-1}(|iGC\cM_n|) \ar@{>->}[d]^{\Delta_n} \\
\pi_r(|qGB\cM_n|) \ar[r]^{\partial} \ar@{->>}[d]^{\partial} & \pi_{r-1}(|iGB^q \cM_n|) \ar[r] \ar@{->>}[d]^{\partial} & \pi_{r-1}(|iGB\cM_n|) \\
\pi_{r-1}(\mathrm{hofib}(\Delta_n)) \ar[r] & K_{r}^\mathrm{Gr}(\cM_n)
}\]
Note that the target of the connecting homomorphism $\partial$ in the middle column of this diagram a priori is the $(r-2)^\mathrm{th}$ homotopy group of the homotopy fibre of the middle vertical morphism $\Delta_n$ in the diagram above; but that homotopy group is isomorphic to $K_r^\mathrm{Gr}(\cM_n)$ due to the identifications
\[\pi_{r-1}(|iGC^q\cM_n|) \cong \pi_{r-1}(|gC^q\cM_n|) = K_{r-1}(C^q\cM_n)\cong K^\mathrm{Gr}_{r-1}(C^q\cM_n)\]
and
\[\pi_{r-1}(|iGB^q\cM_n|) \cong \pi_{r-1}(|gB^q\cM_n|) = K_{r-1}(B^q\cM_n)\cong K^\mathrm{Gr}_{r-1}(B^q\cM_n)\]
given by \cref{Prop: Step 0} and \cite[Corollary~7.4]{Grayson2012}. The right-hand vertical homomorphism in this diagram is bijective because source and target are isomorphic to $\oplus_{i \in \NN} K_{r-1}(\cM_n)$, as explained in the proof of \cite[Theorem~4.3]{Grayson2012}. A straightforward diagram chase then shows that the lower horizontal homomorphism
\begin{equation}\label{Eq: final isom}
\pi_{r-1}(\mathrm{hofib}(\Delta_n)) \rightarrow K_{r}^\mathrm{Gr}(\cM_n)
\end{equation}
is bijective as well. It is also straightfoward to verify that this isomorphism is the same as the one obtained by applying the $r^\mathrm{th}$ homotopy group functor to (the proof of) \cite[Theorem~4.3]{Grayson2012}.

It remains to prove that the following diagram of homomorphisms between homotopy groups commutes:
\[\xymatrix{
\pi_{r-1}(\mathrm{hofib}(\Delta_n)) \ar[r]^{\lambda^k} \ar[d] & \pi_{r-1}(\mathrm{hofib}(\Delta_{nk})) \ar[d] \\
K_r^\mathrm{Gr}(\cM_n) \ar[r]^{\lambda^k} & K_r^\mathrm{Gr}(\cM_{nk})
}\]
By an easy diagram chase, this immediately follows from the statement that the second identification above commutes with $\lambda^k$, which in turn follows from \cref{Prop: Step 0} and the inductive hypothesis applied to the assembly of power operations on the categories $B^q\cM_n$. Note that the proof of \cite[Lemma~3.4(1)]{HKT17} shows that the categories $B^q\cM_n$ are idempotent complete again, as assumed in the first bullet point at the beginning of this section. The assumption in the second bullet point has been formulated in a way so it automatically holds for $\cM_n$ replaced with $B^q\cM_n$ as well.

This finishes the proof of \cref{Thm: main thm},
\end{proof}

Let $X$ be a quasi-compact scheme. Applying the main outcome of \cite{GraysonExterior} to the assembly of power operations given by the usual tensor product and exterior powers in the category $\cP(X)$ of locally free $\cO_X$-modules of finite rank, we obtain the exterior power operation
\[\lambda^k: K_r(X) \rightarrow K_r(X)\]
for any $r \ge 0$ and $k \ge 1$. In the following corollary, we identify $K_r^\mathrm{Gr}(X)$ with $K_r(X)$ via \cite[Corollary~7.4]{Grayson2012}.

\begin{cor}\label{cor: comparing exterior power operations}
Let $r \ge 0$ and $k \ge 1$. The exterior power operation~$\lambda^k$ on $K_r(X)$ constructed in \cite{GraysonExterior} agrees with the one defined in \cite[Theorem~6.2]{HKT17}.
\end{cor}

\begin{proof}
This is a special case of the previous theorem. To see this, we observe that the category $\cP(X)$ is idempotent complete and that the assumption in the second bullet point at the beginning of this section is satisfied by \cref{prop: exterior powers of quasi-isom}. Note that the map $\lambda^k: K_r^\mathrm{Gr}(\cP(X)) \rightarrow K_r^\mathrm{Gr}(\cP(X))$ is indeed the same as the one constructed and denoted $\lambda^k: K_r(X) \rightarrow K_r(X)$ in \cite[Theorem~6.2]{HKT17}.
\end{proof}

\section{Comparison with other operations}\label{sec: comparison with other operations}

In this section we apply the method of \cite[Section~5.2.2]{Zan19} to compare the exterior power operations considered in \cref{cor: comparing exterior power operations} with the ones defined by Zanchetta \cite{Zan21}, Riou \cite{RRR}, Gillet-Soul\'e \cite{GS} and Barwick-Glasman-Mathew-Nikolaus \cite{Barwick}.

\medskip

All schemes considered in this section are assumed to be Noetherian and of finite Krull dimension. We fix a divisorial (\cite[2.1.1]{ThomTro}) regular base scheme $S$.

\medskip

Let $\schs$ denote the category of divisorial schemes of finite type over $S$ and let $\pre(\schs)$ (resp.\ $\spre(\schs)$) denote the category of presheaves (resp.\ simplicial presheaves) of sets over $\schs$. We define $K \in \spre(\schs)$ to be the simplicial presheaf which associates the simplicial set $g\cP(X)$ (see \cref{sec: operations on spaces}) to any scheme~$X$. We view $\spre(\schs)$ as a model category where the cofibrations and the weak equivalences are the sectionwise cofibrations and weak equivalences, respectively. This is called the \emph{global injective model structure}. When applied to the standard assembly of power operations on $\cP(X)$ for any $X \in \schs$, the operations $\lambda^k$ constructed in \cite{GraysonExterior} and recalled in \cref{sec: operations on spaces} are functorial in $X$ and therefore define morphisms $\lambda^k \colon K\rightarrow K$
in $\Ho(\spre(\schs))$, the associated homotopy category.

As the Zariski topology is a Grothendieck topology on $\schs$, we can define a new model category structure on $\spre(\schs)$ having the same cofibrations but where the weak equivalences are the so-called local weak equivalences with respect to the Zariski topology (for more details see \cite[Sections~2--4]{J15}). We
denote the resulting model category by $\spre_\mathrm{Zar}(\schs)$. As this model structure is obtained as a left Bousfield localisation from the global injective one (see \cite[Remark~3.1.4]{AHW17}),
we get a functor $\Ho(\spre(\schs))\rightarrow \Ho(\spre_\mathrm{Zar}(\schs))$ and our operations $\lambda^k$ are sent by this morphism to operations
\[\lambda^k:K\rightarrow K\quad\mathrm{in}\quad \Ho(\spre_\mathrm{Zar}(\schs)).\]

Note all the definitions and constructions given so far in this section also apply to the full subcategory $\sms\subseteq\schs$ of smooth schemes.

Recall that, for divisorial schemes, Thomason's $K$-theory of perfect complexes is weakly equivalent to the $g$-construction, see \cite[Theorem~3.1]{GG} and \cite[Proposition~3.10]{ThomTro}. In the following theorem, we accordingly identify the object $K \in \Ho(\spre_\mathrm{Zar}(\schs))$ with the $K$-theory presheaf as defined in \cite[Section~2]{Zan21}.

\begin{thm}\label{thm: comparison with Zanchetta and Riou}
Let $k \ge 1$. The operation $\lambda^k\in\Homm_{\Ho(\spre_\mathrm{Zar}(\schs))}(K,K)$ agrees with the one defined in \cite[Theorem~7.1]{Zan21}. In particular, for any $X \in \sms$ and $r \ge 0$, the operation $\lambda^k$ on $K_r(X)$ considered in \cref{cor: comparing exterior power operations} agrees with the one defined in \cite[Theorem~1.1.1]{RRR}.
\end{thm}
\begin{proof}
By viewing a scheme as a constant simplicial presheaf using the Yoneda embedding, we obtain the presheaf
\[\pi_0 K:=\Homm_{\Ho(\spre_\mathrm{Zar}(\schs))}(-,K)\in\pre(\schs).\]
As the $K$-theory of perfect complexes satisfies Zariski descent (see \cite[Theorems~8.1 and~10.8]{ThomTro}), we have that $\pi_0K$ coincides with~$K_0$.
Furthermore, by construction and by \cite[Section~8]{GraysonExterior}, the operation $\pi_0(\lambda^k)$ then coincides with the usual lambda operation $\lambda^k \colon K_0\rightarrow K_0$.
The statement now follows from \cite[Theorems~3.2 and~4.1]{Zan21}.
\end{proof}

\begin{cor}\label{cor: comparing with Gillet-Soule}
For any $X \in \schs$ and any $r \ge 0$ and $k \ge 1$, the operation~$\lambda^k$ on $K_r(X)$ considered in \cref{cor: comparing exterior power operations} agrees with the ones constructed in \cite[Theorem~3]{GS} and in \cite[Theorem~1.1]{Barwick}.
\end{cor}

\begin{proof}
Recall that $K \cong\mathbb{Z}\times\mathrm{BGL}^+$ in $\Ho(\spre_\mathrm{Zar}(\schs))$, see \cite[Theorem 3.1]{GG} and \cite[Lem\-ma~18]{GS}. Although a different terminology is used in \cite{GS}, their operations arise from maps $\mathbb{Z}\times\mathrm{BGL}^+\rightarrow\mathbb{Z}\times\mathrm{BGL}^+$ in $\Ho(\spre_\mathrm{Zar}(\schs))$ inducing on $K_0$ the usual lambda operation $\lambda^k$. Hence, the previous theorem and \cite[Theorems 3.2 and 4.1]{Zan21} imply that the operation~$\lambda^k$ on $K_r(X)$ considered in \cref{cor: comparing exterior power operations} agrees with the one constructed in \cite[Theorem~3]{GS}, as claimed.

When applied to the polynomial functor $\lambda^k$ on the category $\mathrm{Perf}(X)$ of perfect complexes on $X$, Theorem~1.1 in \cite{Barwick} yields an operation on the $K$-theory space of $X$ which is functorial in $X$ and which agrees with the usual operation $\lambda^k$ on $K_0(X)$ after passing to $\pi_0$. In particular, we obtain an operation $\lambda^k$ on $K \in \Ho(\spre(\schs))$ and then on $K \in \Ho(\spre_\mathrm{Zar}(\schs))$ as above. Hence, the same reasoning as above shows that the operation~$\lambda^k$ on $K_r(X)$ considered in \cref{cor: comparing exterior power operations} agrees with the one constructed in \cite[Theorem~1.1]{Barwick} as well.
\end{proof}

\begin{rem}
The operations defined in \cite{GS} (and also explained in \cite[Appendix~B]{HW}) are generalisation of the ones defined in \cite[Proposition~4]{Sou85} for regular schemes which in turn agree with the ones defined by Hiller \cite{Hiller} and Kratzer \cite{Kratzer} for affine schemes, see \cite[end of p.~512]{Sou85}.
\end{rem}

\section{The \texorpdfstring{$\lambda$}{lambda}-Ring Axioms for Higher Equivariant \texorpdfstring{$K$}{K}-Theory}\label{sec: Equivariant K-theory}

In \cite{GRR}, the first author applied Grayson's construction \cite{GraysonExterior} to define exterior power operations on higher \emph{equivariant} $K$-theory and conjectured that composition of these operations satisfies the relevant $\lambda$-ring axiom. The object of this section is to prove this conjecture and to thereby finish the proof that Grayson's operations define a $\lambda$-ring structure on the higher equivariant $K$-groups. To this end, we will apply \cref{Thm: main thm} and extend the reasoning of \cite[Section~8]{HKT17} to the equivariant setting. We begin by recalling some notations and facts from \cite{GRR}.

\medskip

Let $S$ be a Noetherian scheme, and let $G$ be a \emph{flat} group scheme over $S$. As usual, a \emph{$G$-scheme over $S$} is an $S$-scheme equipped with an $S$-morphism $m_X\colon G \times X \rightarrow X$ which is associative in the obvious sense. We fix $S$, $G$ and a $G$-scheme $X$ throughout this section.

\medskip

\newcommand\DistTo{\xrightarrow{
   \,\smash{\raisebox{-0.65ex}{\ensuremath{\scriptstyle\sim}}}\,}}

A \emph{$G$-module on $X$} is a pair consisting of an $\cO_X$-module $V$ and an isomorphism $m_V \colon m_X^* V \DistTo \mathrm{pr}_X^* V$ of $\cO_{G\times X}$-modules which satisfies the associativity property
\[(\mathrm{pr}_{2,3}^* m_V) \circ ((1 \times m_X)^*m_V) = (m_G \times 1)^* m_V.\]
Here, $m_G \colon G \times G \rightarrow G$ is the multiplication, and $\mathrm{pr}_X \colon G \times X \rightarrow X$ and $\mathrm{pr}_{2,3} \colon G\times G \times X \rightarrow G \times X$ denote the obvious projections. For example, if $G = \coprod_{\gamma \in \Gamma} S$ is the constant group scheme associated to an abstract group~$\Gamma$, then a $G$-module on $X$ is the same as an $\cO_X$-module $V$ together with isomorphisms $\gamma^* V \DistTo V$, $\gamma \in \Gamma$, which satisfy the usual associativity property. Homomorphisms between $G$-modules are defined in the obvious way. The category of $G$-modules which are coherent as $\cO_X$-modules (resp.\ locally free of finite rank over~$\cO_X$) will be denoted $\cM(G,X)$ (resp. $\cP(G,X)$). By \cite[Lemma~(1.3)]{GRR}, the category $\cM(G,X)$ is an abelian category and the subcategory $\cP(G,X)$ is an exact category.

\medskip

The ordinary non-equivariant tensor product and exterior powers of $G$-modules on~$X$ when equipped with the ``diagonal'' action of $G$ define an assembly of power operations on $\cP(G,X)$ as in the non-equivariant context, see explanations after \cref{def: power operations}. As recalled in \cref{sec: operations on spaces}, from \cite{GraysonExterior} we hence obtain an exterior power operation
\[\lambda^k \colon K_r(G,X) \rightarrow K_r(G,X)\]
on the $r^\mathrm{th}$ \emph{equivariant $K$-group} $K_r(G,X) := K_r(\cP(G,X))$ for any $k \ge 1$ and $r \ge 0$. Conjecture~(2.7) in \cite{GRR} predicts that composition of these exterior power operations obeys the usual lambda-ring axiom for composition, see \cite[I \S 1]{FL}. The following theorem proves this conjecture.

\medskip

Let $P_{k,l}$ denote the universal integral polynomial in $kl$ variables defined for example on page~5 of \cite{FL}. It is easy to derive from its definition that $P_{k,l}$ is homogeneous of degree $kl$ and hence just a multiple of $\lambda^{kl}$ if we discard all monomials in $P_{k,l}$ that are made up of more than one factor. When $r > 0$, this simplifies the right-hand side of the formula in the next theorem, as we define the multiplication on $K_r(G,X)$ to be the trivial one, as is standard in this context. More precisely, let $K(G,X) := \oplus_{r \ge 0} K_r(G,X)$; then, the product~$xy$ of any two elements $x \in K_r(G,X)$ and $y \in K_s(G,X)$ is induced by the tensor product if $r=0$ or $s=0$ and defined to be zero if $r>0$ and $s>0$. (The reader interested in a formula for $\lambda^k$ applied to the \emph{external} product $x \smallsmile y \in K_{r+s}(G,X)$ is referred to \cite[Theorem~1.5]{HK}.)

\medskip

Moreover, $\lambda^k$ canonically extends to the whole direct sum~$K(G,X)$, see \cite[Section~7]{HKT17}. Together with \cite[Propostion~(2.5)]{GRR}, the following theorem proves that $K(G,X)$ is a $\lambda$-ring (see \cite{FL}).

\begin{thm}\label{thm: conjecture}
Let $k,l \ge 1$, $r \ge 0$ and $x \in K_r(G,X)$. Then we have
\begin{equation}\label{eq: third lambda ring axiom}
\lambda^k (\lambda^l(x)) = P_{k,l}(\lambda^1 (x), \ldots, \lambda^{kl}(x))\quad \textrm{in} \quad K_r(G,X).
\end{equation}
\end{thm}

\begin{proof}
We put $K_0^\mathrm{Gr}(G,X) := K_0(G,X) = K_0(\cP(G,X))$ and inductively define
\[K_r^\mathrm{Gr}(G,X) := \mathrm{coker}\left(K_{r-1}^\mathrm{Gr} (C^q\cP(G,X)) \xrightarrow{\Delta} K_{r-1}^\mathrm{Gr} (B^q\cP(G,X))\right)\]
for $r \ge 1$. As explained prior to \cref{Thm: main thm}, we obtain an operation
\[\lambda^k \colon K_r^\mathrm{Gr}(G,X) \rightarrow K_r^\mathrm{Gr}(G,X)\]
for any $k \ge 1$ and $r \ge 0$. By \cref{Thm: main thm}, the isomorphism $K_r(G,X) \cong K_r^\mathrm{Gr}(G,X)$ from \cite[Corollary 7.4]{Grayson2012} commutes with $\lambda^k$. Indeed the assumptions at the beginning of \cref{sec: comparing lambda operations} are satisfied: the assumption in the first bullet point is evident and the assumption in the second bullet point follows from \cref{prop: exterior powers of quasi-isom} as this assumption refers to the underlying $\cO_X$-modules and does not involve the equivariant structures.

It therefore remains to prove \cref{thm: conjecture} when $K_r(G,X)$ is replaced with $K_r^\mathrm{Gr}(G,X)$. To this end, we follow the proof of \cite[Theorem~8.18]{HKT17}. In a nutshell, this means we continue reducing \cref{thm: conjecture} to an analogous statement by successively replacing $K_r(G,X)$ with some other $K$-group until we are arrive at one for which we can prove the required identity. The next paragraph will recap these reductions; for more details, the reader is referred to \emph{loc.\ cit.} The last paragraph will give a detailed proof of that reduction which needs additional reasoning in the equivariant setup of \cref{thm: conjecture}.

By definition, we have an epimorphism $K_0((B^q)^r\cP(G,X)) \rightarrow K_r^\mathrm{Gr}(G,X)$; thus, we may replace $K_r^\mathrm{Gr}(G,X)$ with $K_0((B^q)^r\cP(G,X))$, which actually yields a more precise statement because the multiplication in the latter $K$-group is no longer trivial. Then we reduce to the case when $x$ is the class of an object $P \in  (B^q)^r\cP(G,X)$. Plugging in $P$ shifts our statement to proving it when $(B^q)^r\cP(G,X)$ is replaced with $\mathrm{End}\left((B^q)^r\cP(G,X)\right)$, the category of endo-functors; here, the required structure, such as tensor products and exterior products on $\mathrm{End}\left((B^q)^r\cP(G,X)\right)$, is induced from the target $(B^q)^r\cP(G,X)$ in $\mathrm{End}\left((B^q)^r\cP(G,X)\right)$ in the usual way. Let $\mathrm{Pol}^0_{< \infty}(\ZZ)$ denote the exact category of polynomial functors $F$ that are a finite direct sums of homogeneous polynomial functors over $\ZZ$ and that satisfy $F(0)=0$, see \cite[Section~8A]{HKT17}. The crucial point now is that, as in the non-equivariant case, we have an exact functor
\begin{equation}\label{eq: functor from Pol}
\mathrm{Pol}^0_{<\infty}(\ZZ) \rightarrow \mathrm{End}\left((B^q)^r\cP(G,X)\right)
\end{equation}
which respects tensor products and exterior powers, see below. This  reduces our problem to the case when $\mathrm{End}\left((B^q)^r\cP(G,X)\right)$ is replaced with $\mathrm{Pol}^0_{< \infty}(\ZZ)$, in which case, by \cite[Theorem~8.5]{HKT17}, the Grothendieck group  $K_0\left(\mathrm{Pol}^0_{< \infty}(\ZZ)\right)$ is naturally isomorphic to the augmentation ideal in the universal $\lambda$-ring $\ZZ[s, \lambda^2(s), \ldots]$ in one variable $s$. Hence, our statement holds.

We still need to construct the functor \eqref{eq: functor from Pol}. To this end, recall from \cite[Section~8A]{HKT17} that a polynomial functor $F$ over a scheme $Y$ assigns an object~$F(V)\in \cP(Y)$ to any object~$V \in \cP(Y)$ and a homomorphism $F(\alpha)\colon F(V)_T \rightarrow F(W)_T$ to any homomorphism $\alpha \colon V_T \rightarrow W_T$ for any $V, W \in \cP(Y)$ and for any $Y$-scheme~$T$; the latter assignments must satisfy the usual functor axioms and be functorial in $T$, but need not be linear. Furthermore, for any morphism $Y' \rightarrow Y$ of schemes, there is a base change functor $\mathrm{Pol}_{< \infty}^0(Y) \rightarrow \mathrm{Pol}_{< \infty}^0(Y')$, see the construction in \emph{loc.\ cit.} The base change of any $F \in \mathrm{Pol}_{<\infty}^0(Y)$ will be denoted $F$ again and satisfies $F(f^*(V)) \cong f^*(F(V))$ by construction.  Now, given $F \in \mathrm{Pol}_{<\infty}^0(\ZZ)$ and $V \in \cP(G,X)$, we base-change $F$ from $\ZZ$ to $X$ and turn the $\cO_X$-module~$F(V)$ into a $G$-module via the isomorphism
\[\xymatrix{m^*_X F(V) \cong F(m^*_X V) \ar[r]^{F(m_V)} & F(p^*_X V) \cong p^*_X F(V)}.\]
This defines the functor \eqref{eq: functor from Pol} when $r=0$. The endo-functor of $\cP(G,X)$ just constructed sends $0$ to $0$ again and it commutes with restriction to any open subset $U$ of $X$ (after defining it for $X$ replaced with $U$ as well). As explained in \cite[Sections~4 and 6]{HKT17} (alternatively, in \cref{cor: Dold-Puppe operations IV} and \cref{prop: exterior powers of quasi-isom}), any such functor induces an endo-functor of $(B^q)^r\cP(G,X)$ for every $r \ge 0$. We have thus finally constructed the functor~\eqref{eq: functor from Pol}; it is straightforward to check that it satisfies all the required properties.
\end{proof}

\appendix

\section{\texorpdfstring{$K$}{K}-Theory of Non-negatively Supported Complexes}\label{Sec: non-negative complexes}
In this appendix we shall prove the relatively elementary fact that, in many situations, considering arbitrary bounded complexes instead of complexes supported in non-negative degrees doesn't result in any difference when passing to $K$-theory. An earlier statement of this kind is \cite[Proposition~1.4]{HKT17}.

\medskip

Let $\cM$ be an exact category and consider one of the categories $C\cM$, $C^\chi\cM$, $B\cM$, $B^b\cM$ as defined in the paper. We denote this category by $\cC_{\ge 0}$ in this appendix. The analogous category defined using complexes supported in degree $\geq -i$ (for any $i\in\mathbb{N}$) or using arbitrary bounded complexes is denoted by $\cC_{\geq -i}$ and $\cC$, respectively. We equip $\cC_{\ge -i}$ and $\cC$ with the subcategory~$q$ (or $b$ in case of $B\cM$) of weak equivalences as in the paper and denote the resulting
$K$-theory spaces (defined via Waldhausen's $S.$-construction as in \cite[Section~1.3]{Wald} or via the equivalent $G$-construction as explained in \cref{sec: operations on spaces}) by $K(q\cC_{\ge -i})$ and $K(q\cC)$, respectively.

\begin{prop}\label{prop: bounded vs nonnegative complexes}
The canonical map $K(q\cC_{\geq 0}) \rightarrow K(q\cC)$ is a homotopy equivalence.
\end{prop}

\begin{proof}
We have inclusions $\cC_{\geq -i}\rightarrow\cC_{\geq -i-1}$ that can be bundled together in a filtered system indexed by $\mathbb{N}$ whose colimit is $\cC$ and that preserve the subcategories $q$. By its construction, the $K$-theory space functor commutes with filtered colimits. It therefore suffices to show that the induced map $K(q\cC_{\geq -i})\rightarrow K(q\cC_{\geq -i-1})$ is a homotopy equivalences for every $i \ge 0$. To this end, we consider the functor $[-1]$ that shifts any complex by $1$ to the left and multiplies its differential by $-1$. For every object $V. \in \cC_{\ge -i}$, the standard short exact sequence
\[0\rightarrow V. \rightarrow \mathrm{cone}(\id_{V.})\rightarrow V.[-1]\rightarrow 0\]
is then a sequence within $\cC_{\ge -i}$. (Note the cone construction also applies to binary complexes.) Hence, by the Additivity Theorem (see \cite[Theorem~1.4.2 and Proposition 1.3.2(4)]{Wald} or \cite[Corollary~1.7.3]{ThomTro}), we have $K([-1]) + K(\id) \simeq K(F)$ on $K(q\cC_{\ge -i})$  where $F$ is the endofunctor $V. \mapsto \mathrm{cone}(\mathrm{id}_{V.})$ on $\cC_{\ge -i}$. Filtering the double complex underlying $\mathrm{cone}(\id_{V.})$ vertically rather than horizontally, we similarly obtain that $K(F) \simeq \sum_{j\ge 0} K(F_j)$  where the functor $F_j$ is given by $V. \mapsto \mathrm{cone}(\id_{V_j[-j]})$ (see proof of \cite[Lemma~1.6]{HKT17} for details).  We furthermore have $K(F_j) \simeq 0$ as $\mathrm{cone}(\id_{V_j[-j]})$ is naturally homotopy equivalent to the zero complex. It remains to observe that composing the inclusion $\cC_{\geq -i}\rightarrow\cC_{\geq -i-1}$ in any direction with the shift functor $[-1] \colon \cC_{\geq -i-1} \rightarrow \cC_{\geq -i}$ is the endofunctor~$[-1]$ considered above.
\end{proof}

\section{Homotopy Fibres and Homotopy Fibre Sequences}\label{Sec: Homotopy}

In this appendix, we first give a precise definition of a (homotopy) fibre sequence and then detail a folklore fact about morphisms between fibre sequences that is crucial for the proof of \cref{Prop: identifying lambda operations} and that is in fact at the heart of the proof of our main theorem (\cref{Thm: main thm}). Standard references are \cite{Mather} and \cite{MuVoCu}.

\medskip

Given a topological space $Y$, let $Y^I$ denote the space of continuous maps from the unit interval $I$ to $Y$. Given a pointed space $(Y,y)$, we define the \emph{path space} $PY:=\lbrace \alpha\in Y^I\: |\: \alpha(0)=y\rbrace$. We also have a continuous map $\mathrm{ev}_1 \colon PY\rightarrow Y$ evaluating every path at $1$.

\begin{defn}[{\cite[2.2.3]{MuVoCu}}]
Given a continuous map $f:X\rightarrow Y$ and a point $y\in Y$, the \emph{homotopy fibre} of $f$ (at $y$) is the topological space
\[\mathrm{hofib}(f):=\lbrace (x,\alpha)\in X\times Y^I \mid f(x)= \alpha(1),\: \alpha(0)=y \rbrace.\]
\end{defn}

In other words, $\mathrm{hofib}(f)$ is the fibred product $X \times_Y PY$ of $X$ and $PY$ over $Y$ with respect to the maps $f \colon X \rightarrow Y$ and $\mathrm{ev}_1 \colon PY \rightarrow Y$.

Consider a sequence of continuous maps $W\xrightarrow{g}X\xrightarrow{f}Y$ and a point $y\in Y$. Let  $p \colon \mathrm{hofib}(f)\rightarrow X$ denote the map given by $(x,\alpha)\mapsto x$ and let $c_y \colon W \rightarrow Y$ denote the constant map with value $y$. It can be verified that there is a canonical bijection $H \mapsto \widetilde{H}$ between the set of homotopies $H \colon c_y \sim f\circ g$ and the set of continuous maps $G \colon W\rightarrow\mathrm{hofib}(f)$ such that $p\circ G =g$. The map $\widetilde{H}$ is called the \emph{whisker map}.

\begin{defn}\label{def: null-homotopy}
We say that a sequence $W\xrightarrow{g}X\xrightarrow{f} Y$ of topological spaces together with a distinguished point $y \in Y$ and a homotopy $H \colon c_y \sim f\circ g $ is a \emph{sequence with null-homotopy}. If the corresponding whisker map $\widetilde{H}$ is a homotopy equivalence, we say that it is a \emph{(homotopy) fibre sequence}.
\end{defn}

For example, given a continuous map $f \colon X \rightarrow Y$ and $y \in Y$, the sequence $\mathrm{hofib}(f)\xrightarrow{p}X\xrightarrow{f}Y$ is a fibre sequence. It is called the \emph{standard fibre sequence} associated with $f$ and $y$.

\begin{defn}
As usual, a \emph{morphism $\lambda: (X \xrightarrow{f} Y) \rightarrow (X' \xrightarrow{f'} Y')$ between continuous  maps $f$ and $f'$} is a pair $\lambda=(\lambda_X, \lambda_Y)$ of continuous maps $\lambda_X \colon X \rightarrow X'$ and $\lambda_Y \colon Y \rightarrow Y'$ such the obvious resulting square commutes. A \emph{morphism $\lambda$ between sequences $W \rightarrow X \rightarrow Y$ and $W' \rightarrow X' \rightarrow Y'$} of topological spaces is defined similarly. When these sequences are sequences with null-homotopies $H$ and $H'$ (or in fact fibre sequences), we say that $\lambda$ is a \emph{morphism of sequences with null-homotopies (or of fibre sequences)} if $\lambda_Y(y) = y'$ and the following diagram commutes:
\begin{equation*}\label{Dia3}
\xymatrix{W \ar[r]^H\ar[d]^{\lambda_W} & PY\ar[d]^{P\lambda_{Y}}\\
W' \ar[r]^{H'} & PY'
}
\end{equation*}
If this diagram commutes only up to homotopy, we say that $\lambda$ is a \emph{weak morphism of sequences with null-homotopies} (or \emph{of fibre sequences}).
\end{defn}

For example, if $\lambda = (\lambda_X, \lambda_Y) \colon f \rightarrow f'$ is a morphism between two continuous maps $f \colon X \rightarrow Y$ and $f' \colon X' \rightarrow Y'$, then, for every $y \in Y$, the triple $(\mathrm{hofib}(\lambda), \lambda_X, \lambda_Y)$ is a morphism between the standard fibre sequences associated with $(f,y)$ and $(f', \lambda_Y(y))$.

\begin{lem}\label{lem: Fiber compatibility}
Let $\lambda \colon (W \xrightarrow{g} X \xrightarrow{f} Y) \rightarrow (W' \xrightarrow{g'} X' \xrightarrow{f'} Y')$ be a morphism between sequences with null-homotopies $H$ and $H'$. Then the following diagram commutes:
\[\xymatrix{W \ar[r]^{\widetilde{H}\hspace*{1.5em}} \ar[d]^{\lambda_W} & \mathrm{hofib}(f) \ar[d]^{\mathrm{hofib}(\lambda)} \\
W' \ar[r]^{\widetilde{H}'\hspace*{1.5em}} & \mathrm{hofib}(f')
}\]
If $\lambda$ is only a weak morphism, then this diagram commutes up to homotopy.
\end{lem}

Thus, the assumption that $\lambda$ is not only a morphism between sequences but a weak morphism between, say, fibre sequences is enough to be able to conclude that $\lambda_W$ and  $\mathrm{hofib}(\lambda)$ induce the same maps on homotopy groups, after identifying the homotopy groups of $W$ and $W'$ with the homotopy groups of $\mathrm{hofib}(f)$ and $\mathrm{hofib}(f')$, respectively, via the whisker maps.

\begin{proof}
The proof is a straightforward exercise in homotopy theory using the very definition of the standard homotopy fibre and the properties of homotopy pullbacks detailed in \cite{Mather}, for example.
\end{proof}

\begin{cor}\label{Equivfib}\label{Cor: eq of sequences}
Let $\lambda=(\lambda_W, \lambda_X, \lambda_Y)$ be a weak morphism as in the previous lemma and assume that $\lambda_W$, $\lambda_X$ and $\lambda_Y$ are homotopy equivalences. If then one of the sequences is a fibre sequence, so is the other.
\end{cor}

\begin{proof}
This follows immediately from \cref{lem: Fiber compatibility} and \cite[Lemma 6]{Mather}.
\end{proof}

\end{document}